\documentclass[11pt]{article}
\pdfoutput=1
\usepackage{a4wide}

\usepackage{amsmath,amssymb,amsthm}
\usepackage{mathrsfs}

\newtheorem{theorem}{Theorem}[section]
\newtheorem{lemma}[theorem]{Lemma}

\newtheorem{corollary}[theorem]{Corollary}

\theoremstyle{definition}
\newtheorem{definition}[theorem]{Definition}
\newtheorem{example}[theorem]{Example}

\theoremstyle{remark}
\newtheorem{remark}[theorem]{Remark}

\numberwithin{equation}{section}
\let\Im\relax\DeclareMathOperator{\Im}{Im}
\let\Re\relax\DeclareMathOperator{\Re}{Re}
\def\GL{\mathrm{GL}}
\def\SL{\mathrm{SL}}
\def\Li{\mathrm{Li}}
\def\bb{\mathbb{B}}

\allowdisplaybreaks

\usepackage{mathtools}
\mathtoolsset{showonlyrefs=true}

\usepackage{tikz}

\begin{document}
\title{On Arakawa--Kaneko zeta-functions associated with $\GL_2(\mathbb{C})$ and their functional relations}

\author{Yasushi Komori and Hirofumi Tsumura}

\date{\today}

\maketitle

\begin{abstract}
  We construct a certain class of Arakawa--Kaneko zeta-functions associated with $\GL_2(\mathbb{C})$, which includes the ordinary Arakawa--Kaneko zeta-function. We also define poly-Bernoulli polynomials associated with $\GL_2(\mathbb{C})$ which appear in their special values of these zeta-functions. 
We prove some functional relations for these zeta-functions, which 
are regarded as interpolation formulas of various relations among poly-Bernoulli numbers.
Considering their special values, we prove difference relations and duality relations for poly-Bernoulli polynomials associated with $\GL_2(\mathbb{C})$. 
\end{abstract}

\footnotetext{This work was supported by Japan Society for the Promotion of Science, Grant-in-Aid for Scientific Research (C) 25400026 and 15K04788.}

\section{Introduction} \label{sec-1}

For $k\in \mathbb{Z}$, two types of poly-Bernoulli numbers $\{B_n^{(k)}\}$ and $\{C_n^{(k)}\}$ are defined by Kaneko as follows:
\begin{align}
&\frac{\Li_{k}(1-e^{-t})}{1-e^{-t}}=\sum_{n=0}^\infty B_n^{(k)}\frac{t^n}{n!},  \label{1-1}\\
&\frac{\Li_{k}(1-e^{-t})}{e^t-1}=\sum_{n=0}^\infty C_n^{(k)}\frac{t^n}{n!},  \label{1-2}
\end{align}
where $\Li_{k}(z)$ is the polylogarithm defined by
\begin{equation}
\Li_{k}(z)=\sum_{m=1}^\infty \frac{z^m}{m^k}\quad (|z|<1) \label{1-3}
\end{equation}
(see Kaneko \cite{Kaneko1997} and Arakawa--Kaneko \cite{AK1999}, also Arakawa--Ibukiyama--Kaneko \cite{AIK2014}). Since $\Li_1(x)=-\log(1-x)$, we see that $B_n^{(1)}$ coincides with the ordinary Bernoulli number. 

In this decade, these numbers have been actively investigated (see, for example, Kaneko \cite{Kaneko-Mem}). The most remarkable formulas for them are the following `duality relations':
\begin{align}
& B_m^{(-k)}=B_{k}^{(-m)},\label{1-5}\\
& C_m^{(-k-1)}=C_{k}^{(-m-1)} \label{1-5-2}
\end{align}
for $k,m\in \mathbb{Z}_{\geq 0}$ (see \cite[Theorem 2]{Kaneko1997} and \cite[\S\,2]{Kaneko-Mem}). Recently Kaneko and the second-named author \cite{Kaneko-Tsumura2015} showed \eqref{1-5}, \eqref{1-5-2} and their generalization by investigating the zeta-function of Arakawa--Kaneko type (defined below). 
Also it is known that 
\begin{equation}
B_m^{(k)}=C_{m}^{(k)}+C_{m-1}^{(k-1)} \label{1-5-3}
\end{equation}
for $k\in \mathbb{Z}$ and $m\in \mathbb{Z}_{\geq 1}$ (see \cite[Equation (9)]{AK1999}).

Corresponding to these numbers, Arakawa and Kaneko defined the zeta-function 
\begin{equation}
\xi(k;s)=\frac{1}{\Gamma(s)}\int_0^\infty {t^{s-1}}\frac{\Li_{k}(1-e^{-t})}{e^t-1}dt  \quad (\Re s>0) \label{1-9}
\end{equation}
for $k \in {\mathbb{Z}}_{\geq 1}$, 
which can be continued to $\mathbb{C}$ as an entire function (see \cite[Section 3]{AK1999}). Further they considered multiple versions of \eqref{1-9}. Note that $\xi(k;s)$ can be regarded as generalizations of the Riemann zeta-function because $\xi(1;s)=s\zeta(s+1)$. They also showed that
\begin{equation}
\xi(k;-m)=(-1)^m C_m^{(k)} \quad (m\in \mathbb{Z}_{\geq 0}). \label{1-10}
\end{equation}
From the observation of $\xi(k;s)$ and its multiple versions, they gave several relation formulas among the multiple zeta values defined by
\begin{equation}
\zeta(l_1,\ldots,l_r)=\sum_{1<m_1<\cdots<m_r}\frac{1}{m_1^{l_1}\cdots m_r^{l_r}}
\label{MZV}
\end{equation}
for $l_1,\ldots,l_r\in \mathbb{Z}_{\geq 1}$ with $l_r\geq 2$ 
(see \cite[Corollary 11]{AK1999}). 

As a generalization of $\xi(k;s)$, Coppo and Candelpergher \cite{CC2010} defined 
\begin{equation}
\xi(k;s;w)=\frac{1}{\Gamma(s)}\int_0^\infty {t^{s-1}}e^{-wt}\frac{\Li_{k}(1-e^{-t})}{1-e^{-t}}dt  \label{CC}
\end{equation}
for $k\in \mathbb{Z}_{\geq 1}$ and $w>0$, and studied its property. Note that $\xi(k;s;1)=\xi(k;s)$. 

As a twin sibling of \eqref{1-9}, Kaneko and the second-named author \cite{Kaneko-Tsumura2015} recently defined 
\begin{equation}
\eta(k;s)=\frac{1}{\Gamma(s)}\int_{0}^\infty t^{s-1}\frac{\Li_{k}(1-e^t)}{1-e^t}dt \label{1-11}
\end{equation}
for $s\in \mathbb{C}$ and for `any' $k\in \mathbb{Z}$, 
which interpolates the 
poly-Bernoulli numbers of $B$-type, that is,
\begin{equation}
\eta(k;-m)=B_m^{(k)}\qquad (k\in \mathbb{Z},\ m\in \mathbb{Z}_{\geq 0}). \label{B-interpolate}
\end{equation}
More generally, they defined the multi-variable version of 
\eqref{1-11} denoted by $\eta((-k_j);(s_j))$ for each $k_j\in \mathbb{Z}_{\geq 0}$, 
and showed certain duality relations for multi-indexed poly-Bernoulli numbers 
(see \cite[Theorem 5.7 and 5.10]{Kaneko-Tsumura2015}). 

More recently, Yamamoto \cite{Yamamoto} 
considered $\eta(u;s)$ (where $u$ and $s$ are variables) and its multi-variable versions $\eta((u_j);(s_j))$ 
and proved functional duality relations for them. In particular, for the case of single zeta-function, he proved
\begin{equation}
\eta(u;s)=\eta(s;u) \quad (u,s\in \mathbb{C}),\label{Yamamoto-F}
\end{equation}
which interpolates \eqref{1-5} at non-positive integer points  by \eqref{B-interpolate}. 

In this paper, we consider, as generalizations of $\xi(k;s)$, $\eta(k;s)$ and $\xi(k;s;w)$, the Arakawa--Kaneko zeta-functions associated with $\GL_2(\mathbb{C})$ defined as follows. 
For $g=\begin{pmatrix}a & b\\ c& d \end{pmatrix}\in\GL_2(\mathbb{C})$, we let
\begin{equation*}
gz=\frac{az+b}{cz+d},\ \ j_D(g,z)=cz+d, \ \ j_N(g,z)=az+b. 
\end{equation*}
Note that $j_D(g,z)$ coincides with the factor of automorphy for $g\in\SL_2(\mathbb{Z})$ (see \cite[\S\,1.2]{Diamond2005}). 
Let 
\begin{equation}
  \Phi(z,u,y)=  \sum_{m=0}^\infty\frac{z^m}{(m+y)^u}  \label{Lerch-T}
\end{equation}
be the Lerch Transcendent for $z,u,y\in\mathbb{C}$ with $|z|<1$ or ($z=1$ and $\Re u>1$), and $\Re y>0$ (see \cite[\S, 1.11]{Bateman}). For $y,w \in \mathbb{C}$, we define
\begin{equation}
    \xi_D(u,s;y,w;g)=\frac{1}{\Gamma(s)}\int_{0}^\infty t^{s-1}e^{-wt}\frac{\Phi(ge^t,u,y)}{j_D(g,e^t)}dt, \label{Def-main-1}
\end{equation}
which is the main object in this paper. We construct interpolation formulas of the well-known relations among poly-Bernoulli numbers by use of $\xi_D(u,s;y,w;g)$.


In Section \ref{sec-2}, we define the Lerch Transcendent and study its properties and related results.

In Section \ref{sec-3}, we define \eqref{Def-main-1} (see Definition \ref{Def-3-1}) and determine its domain (see Theorem \ref{thm:domain1}).
We confirm 
that $\xi(k;s)$, $\eta(k;s)$ and $\xi(k;s;w)$ can be regarded as special cases of \eqref{Def-main-1} (see Example \ref{Exam-3-1}). 

In Section \ref{sec-4}, we give two types of functional relations among \eqref{Def-main-1} which include \eqref{Yamamoto-F} as a special case 
(see Theorems \ref{thm:trans1} and \ref{thm:trans2}). Combining these formulas, we give interpolation formulas 
of the well-known relations 
including \eqref{1-5}--\eqref{1-5-3} (see Example \ref{Exam-4-6}). 

In Section \ref{sec-5}, we consider the analytic continuation for \eqref{Def-main-1} (see Theorems \ref{thm:domain2}, \ref{thm:domain3} and \ref{AC-uyw}), and introduce several examples of duality relations (see Examples \ref{Exam-5-6} and \ref{Exam-5-6-3}).

In Section \ref{sec-6}, we define the poly-Bernoulli polynomials associated with $\GL_2(\mathbb{C})$ (see Definition \ref{g-Bernoulli}). From the results in Sections \ref{sec-4} and \ref{sec-5}, we give general forms of difference relations and duality relations for them (see Theorems \ref{thm:DR} and \ref{Th-D-F}). 
These include \eqref{1-5}--\eqref{1-5-3} and also the duality relations for poly-Bernoulli polynomials (see Example \ref{Exam-6-6}) given by Kaneko, Sakurai and the second-named author (see \cite{KST2016}). Furthermore, we give new duality relations for certain sums of $C_m^{(-k)}$ (see Example \ref{Exam-C-dual}).

\section{Preliminaries}\label{sec-2}
For $z,u,y\in\mathbb{C}$ with $|z|<1$ or ($z=1$ and $\Re u>1$), and $\Re y>0$,
the Lerch transcendent is defined by 
\begin{equation}
  \Phi(z,u,y)=\sum_{n=0}^\infty\frac{z^n}{(y+n)^u},
\end{equation}
which is a generalization of the polylogarithm defined by
\begin{equation}
\Li_u(z)=    \sum_{n=1}^\infty\frac{z^n}{n^u},
\end{equation}
and is related as
\begin{equation}
  z\Phi(z,u,1)=\Li_u(z). \label{Phi-Li}
\end{equation}
For $k\in\mathbb{Z}_{\geq0}$,
the Lerch transcendent satisfies the following.
\begin{align}
\label{eq:FELT1}
  \Phi(z,u,y)
  &=
  z^k\Phi(z,u,y+k)  +\sum_{n=0}^{k-1}\frac{z^n}{(y+n)^{u}}\\
\label{eq:FELT2}
  &=
  z^{-k}\Phi(z,u,y-k) -\sum_{n=1}^{k}\frac{z^{-n}}{(y-n)^{u}}.
\end{align}
\begin{lemma}
\label{lm:IRLT}
For ($\Re u>0$ and $|z|<1$) or ($\Re u>1$ and $z=1$), and $\Re y>0$,
$\Phi(z,u,y)$ has the integral representation
\begin{equation}
    \Phi(z,u,y)
    =
    \frac{1}{\Gamma(u)}
      \int_0^{\infty} x^{u-1}e^{-yx}\frac{1}{1-z e^{-x}}dx.
\end{equation}
This expression gives
the analytic continuation of
$\Phi(z,u,y)$ for
$z\in\mathbb{C}\setminus [1,+\infty)$, $\Re u>0$ and $\Re y>0$.
\end{lemma}
\begin{proof}
First we assume $|z|<1$ or $z=1$.
By an integral representation of the gamma function
$\Gamma(u)$ for $\Re u>0$, we have
\begin{equation}
  \frac{1}{a^u}=\frac{1}{\Gamma(u)}\int_0^\infty e^{-ax}x^{u-1}dx
\end{equation}
for $\Re a>0$.
For ($\Re u>0$ and $|z|<1$) or ($\Re u>1$ and $z=1$),
by substituting this into the series expression, we obtain
\begin{equation}
  \begin{split}
    \Phi(z,u,y)&=\frac{1}{\Gamma(u)}\sum_{n=0}^\infty
    \int_0^\infty z^n e^{-nx}e^{-yx}x^{u-1}dx  
    \\
    &=\frac{1}{\Gamma(u)}
    \int_0^\infty x^{u-1}e^{-yx}\frac{1}{1-z e^{-x}}dx.
  \end{split}
\end{equation}
By this integral representation, $\Phi(z,u,y)$ is analytically continued 
for $z\in\mathbb{C}\setminus [1,+\infty)$, $\Re u>0$ and $\Re y>0$.
\end{proof}

For a variable $u$, we define a difference operator $D_u$ by
\begin{equation}
  D_uf(u)=f(u+1).
\end{equation}
We also define the Euler operator
\begin{equation}
  \vartheta_z=z\dfrac{\partial}{\partial z}.
\end{equation}

\begin{lemma}
\label{lm:Dd}
\begin{equation}
  (D_u^{-1}-y)\Phi(z,u,y)=\vartheta_z\Phi(z,u,y).
\end{equation}
\end{lemma}
\begin{proof}
By the series expression, we have
\begin{equation}
    \Phi(z,u-1,y)
  =\Bigl(y+z\frac{\partial}{\partial z}\Bigr)\Phi(z,u,y),
\end{equation}
which is rewritten in terms of the difference operator $D_u$.
\end{proof}

\begin{lemma}
\label{lm:rd}
For $n\in \mathbb{Z}_{\geq 0}$, 
\begin{align}
  \Bigl(\prod_{k=1}^n \Bigl(1+\frac{1}{k}\vartheta_z\Bigr)\Bigr)
    \frac{1}{1-ze^{-x}}
    &=\frac{1}{(1-ze^{-x})^{n+1}},
  \\
  \Bigl(\prod_{k=1}^n \Bigl(-1+\frac{1}{k}\vartheta_z\Bigr)\Bigr)
    \frac{ze^{-x}}{1-ze^{-x}}
    &=\Bigl(\frac{ze^{-x}}{1-ze^{-x}}\Bigr)^{n+1}.
\end{align}
\end{lemma}
\begin{proof}
Since
  \begin{equation}
    \begin{split}
      \vartheta_z\frac{1}{(1-ze^{-x})^k}
      &=
      \frac{kze^{-x}}{(1-ze^{-x})^{k+1}}
      \\
      &=
      \frac{k}{(1-ze^{-x})^{k+1}}-
      \frac{k}{(1-ze^{-x})^k},
    \end{split}
  \end{equation}
we have
  \begin{equation}
\Bigl(1+\frac{1}{k}\vartheta_z\Bigr)\frac{1}{(1-ze^{-x})^k}=\frac{1}{(1-ze^{-x})^{k+1}},
  \end{equation}
which yields the first equation.

Similarly
  \begin{equation}
    \begin{split}
      \vartheta_z\Bigl(\frac{ze^{-x}}{1-ze^{-x}}\Bigr)^k
      &=      
      k\Bigl(\frac{ze^{-x}}{1-ze^{-x}}\Bigr)^{k+1}
      +k\Bigl(\frac{ze^{-x}}{1-ze^{-x}}\Bigr)^{k}
    \end{split}
  \end{equation}
implies
  \begin{equation}
\Bigl(-1+\frac{1}{k}\vartheta_z\Bigr)\Bigl(\frac{ze^{-x}}{1-ze^{-x}}\Bigr)^k=\Bigl(\frac{ze^{-x}}{1-ze^{-x}}\Bigr)^{k+1}
  \end{equation}
and the second equation.
\end{proof}

\begin{lemma}
\label{lm:DS}
For $n\in \mathbb{Z}_{\geq 0}$, 
  \begin{gather}
    \frac{1}{n!}\Bigl(\prod_{k=1}^n(D_u^{-1}-y+k)\Bigr)
    \Phi(z,u,y)
 =    \frac{1}{\Gamma(u)}
   \int_0^{\infty} x^{u-1}e^{-yx}\frac{1}{(1-z e^{-x})^{n+1}}dx,
 \\
     \frac{1}{n!}\Bigl(\prod_{k=1}^n(D_u^{-1}-y-k)\Bigr)
    (\Phi(z,u,y)-y^{-u})
 =\frac{1}{\Gamma(u)}
   \int_0^{\infty} x^{u-1}e^{-yx}\Bigl(\frac{ze^{-x}}{1-z e^{-x}}\Bigr)^{n+1}dx.
 \end{gather}
\end{lemma}
\begin{proof}
The results follow from
Lemmas \ref{lm:Dd} and \ref{lm:rd}.
\end{proof}

Let $\widehat{\mathbb{C}}=\mathbb{C}\cup\{\infty\}$ denote the Riemann sphere.
For 
$g=\begin{pmatrix}
  a & b \\ c & d
\end{pmatrix}\in\GL_2(\mathbb{C})$,
we define
the M\"obius transformation
\begin{equation}
gz=\dfrac{az+b}{cz+d}
\end{equation}
for $z\in\widehat{\mathbb{C}}$.
Note that it is well known that 
M\"obius transformations are conformal and map
circular arcs to circular arcs, where circular arcs include line segments.
Let
\begin{equation}
  V(g)=\{g1,g\infty\}\cap\{1,\infty\}
\end{equation}
be the intersection of the extremal points of the two circular arcs $g([1,+\infty])$ and $[1,+\infty]$.

Let
\begin{equation}
    j_D(g,z)=cz+d,\qquad
    j_N(g,z)=az+b
\end{equation}
for $z\in\mathbb{C}$. 
Then for $g,h\in\GL_2(\mathbb{C})$, we have
\begin{align}
    j_D(gh,T)&=j_D(g,hT)j_D(h,T),\\
    j_N(gh,T)&=j_N(g,hT)j_D(h,T).
\end{align}

If two circular arcs intersect at their extremal points, we call such point a vertex. 
Moreover if the vertex angle is zero, then we call the vertex a cusp.

For $Z\in\{1,\infty\}$, we denote $\Tilde{Z}=1/Z\in\{0,1\}$.
Let
\begin{equation}
  W_{a,\epsilon,R}=\{z\in\mathbb{C}~|~0<|z-a|<\epsilon\}\cup
\{z\in\mathbb{R}~|~a<z<R\}
\end{equation}
for $a\geq 0$, $\epsilon,R>0$. We abbreviate $W_{a,\epsilon}=W_{a,\epsilon,+\infty}$.

The following lemmas give certain inequalities under the assumption
that the two circular arcs $g([1,+\infty])$ and $[1,+\infty]$
intersect each other possibly only at their extremal points. 
See Figure \ref{fig:tc} for typical configurations.
These
estimations play important roles when the domains of the main objects
are determined.
Their proofs will be given in Section \ref{sec:pr}.

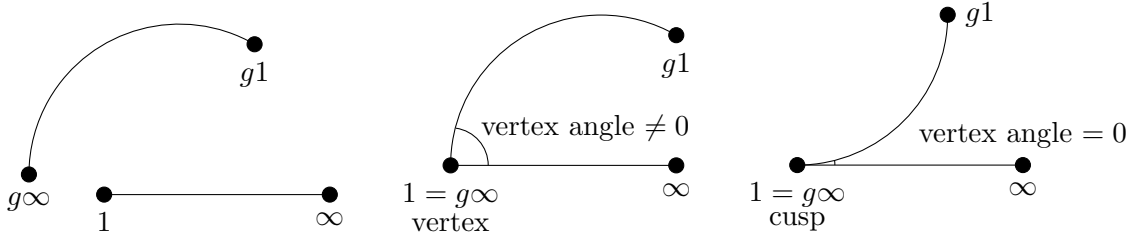
\begin{figure}[h]
  \centering
\begin{tikzpicture}
  \draw (2,2) arc [start angle=60, end angle=180, radius=2] coordinate (A); 
  \draw (0,0) -- (3,0) ; 
  \filldraw (2,2) circle [radius=0.1] node [below=1mm] {$g1$};
  \filldraw (A) circle [radius=0.1] node [below=1mm] {$g\infty$};
  \filldraw (0,0) circle [radius=0.1] node [below=1mm] {$1$};
  \filldraw (3,0) circle [radius=0.1] node [below=1mm] {$\infty$};
\end{tikzpicture}
\quad
\begin{tikzpicture}
  \draw (0,0) arc [start angle=180, end angle=60, radius=2] coordinate (A); 
  \draw (0,0) -- (3,0) ; 
  \node [below=5mm] (0,0) {vertex} ;
  \draw (0:0.5) arc [start angle=0, end angle=82, radius=0.5] node [right=2mm] {vertex angle $\neq 0$}; 
  \filldraw (0,0) circle [radius=0.1] node [below=1mm] {$1=g\infty$};
  \filldraw (A) circle [radius=0.1] node [below=1mm] {$g1$};
  \filldraw (3,0) circle [radius=0.1] node [below=1mm] {$\infty$};
\end{tikzpicture}
\quad
\begin{tikzpicture}
  \draw (0,0) arc [start angle=-90, end angle=0, radius=2] coordinate (A); 
  \draw (0,0) -- (3,0) ; 
  \node [below=5mm] (0,0) {cusp} ;
  \draw (0:0.5) arc [start angle=0, end angle=7, radius=0.5];
  \node at (3,0.4) {vertex angle $=0$} ;
  \filldraw (0,0) circle [radius=0.1] node [below=1mm] {$1=g\infty$};
  \filldraw (A) circle [radius=0.1] node [right=1mm] {$g1$};
  \filldraw (3,0) circle [radius=0.1] node [below=1mm] {$\infty$};
\end{tikzpicture}
\caption{typical configurations}
  \label{fig:tc}
\end{figure}

\begin{lemma}
\label{lm:boundV}
Let $g=\begin{pmatrix}
  a & b \\ c & d
\end{pmatrix}\in\GL_2(\mathbb{C})$ and
 $T_0,X_0\in\{1,\infty\}$.
  Assume that $gT=X$ for only $(T,X)=(T_0,X_0)$ in its neighborhood in $[1,+\infty]^2$.
  \begin{enumerate}
  \item For $0\leq q\leq1$,
there exists $M>0$ such that
\begin{multline}
    \Bigl|\frac{1}{j_D(g,T)}\frac{1}{(1-(gT)X^{-1})}\Bigr|
    \\
    \leq 
  \begin{cases}
    \dfrac{M}{T}\Bigl|\dfrac{T}{\Tilde{T_0}T-1}\Bigr|^{1-q}\Bigl|\dfrac{X}{\Tilde{X_0}X-1}\Bigr|^{q}\qquad&\text{if the vertex is not a cusp}, \\
    \dfrac{M}{T}\Bigl|\dfrac{T}{\Tilde{T_0}T-1}\Bigr|^{2(1-q)}\Bigl|\dfrac{X}{\Tilde{X_0}X-1}\Bigr|^{2q}
    \qquad&\text{if the vertex is a cusp}
  \end{cases}
\end{multline}
in a sufficiently small neighborhood of
$(T_0,X_0)$ in $(1,+\infty)^2$.
\item There exists $\epsilon>0$ such that
\begin{equation}
\label{eq:dist0}
\frac{1}{\epsilon}\Bigl|\dfrac{\Tilde{X_0}X-1}{X}\Bigr|
>\Bigl|\dfrac{\Tilde{T_0}T-1}{T}\Bigr|
>\epsilon \Bigl|\dfrac{\Tilde{X_0}X-1}{X}\Bigr|
\end{equation}
for any 
pair $(T,X)$ satisfying
$gT=X$ in
a sufficiently small neighborhood
of $(T_0,X_0)$ in $\mathbb{C}^2$.
\end{enumerate}
\end{lemma}

\begin{lemma}
\label{lm:boundNV}
Let $g=\begin{pmatrix}
  a & b \\ c & d
\end{pmatrix}\in\GL_2(\mathbb{C})$ 
be such that
\begin{equation}
g([1,+\infty])\cap[1,+\infty]\subset \{g1,g\infty\}\cap\{1,\infty\}=V(g).
\end{equation}
Let 
$N$ be a neighborhood of $\{(T_0,X_0)~|~X_0\in V(g),T_0=g^{-1}X_0\}$ in $\widehat{\mathbb{C}}^2$.
Then
there exist $\epsilon>0$ and $M>0$ such that
\begin{equation}
  \Bigl|\frac{1}{j_D(g,T)}\frac{1}{(1-(gT)X^{-1})}\Bigr|
  \leq \frac{M}{|T|}
\end{equation}
for all $(T,X)\in W_{1,\epsilon}^2\setminus N$.
\end{lemma}

\section{Arakawa--Kaneko zeta-functions associated with $\GL_2(\mathbb{C})$} \label{sec-3}

Here and hereafter
we only consider $g\in\GL_2(\mathbb{C})$ which satisfies that 
\begin{equation}
\label{eq:def_cond}
  g([1,+\infty])\cap[1,+\infty]\subset \{g1,g\infty\}\cap\{1,\infty\}=V(g).
\end{equation}
In this section, 
we give the definition of generalizations of the Arakawa--Kaneko zeta-function. The domains of the functions will be given later, which depend on the configuration of the three points $\{g0,g1,g\infty\}$ on the Riemann sphere $\widehat{\mathbb{C}}$.

\begin{definition}\label{Def-3-1}
For $g\in\GL_2(\mathbb{C})$ satisfying \eqref{eq:def_cond}, 
we define the Arakawa--Kaneko zeta-function associated with $g$ by
\begin{equation}
\label{eq:IR1}
    \xi_D(u,s;y,w;g)=\frac{1}{\Gamma(s)}\int_0^{\infty} t^{s-1}e^{-wt}\frac{\Phi(ge^t,u,y)}{j_D(g,e^t)}dt.
  \end{equation}
We define an auxiliary function
  \begin{equation}
    \xi_N(u,s;y,w;g)=\xi_D(u,s;y+1,w;g).
  \label{xi-D-N}
  \end{equation}
\end{definition}

We have the following integral representation of $\xi_{N}(u,s;y,w;g)$, which clarifies the meaning of the subscripts ``$D$'' and ``$N$''.
\begin{lemma}
\label{lm:DN}
  \begin{equation}
\label{eq:IR2}
      \xi_{N}(u,s;y,w;g)
      =\frac{1}{\Gamma(s)}\int_0^{\infty} t^{s-1}e^{-wt}\frac{(\Phi(ge^t,u,y)-y^{-u})}{j_N(g,e^t)}dt.
\end{equation}
\end{lemma}
\begin{proof}
Since
  \begin{equation}
  z\Phi(z,u,y+1)=\Phi(z,u,y)-y^{-u},
\end{equation}
we have
  \begin{equation}
    \begin{split}
      \xi_{N}(u,s;y,w;g)
      &=\frac{1}{\Gamma(s)}\int_0^{\infty} t^{s-1}e^{-wt}\frac{ge^t\Phi(ge^t,u,y+1)}{j_N(g,e^t)}dt
      \\
      &=\frac{1}{\Gamma(s)}\int_0^{\infty} t^{s-1}e^{-wt}\frac{(\Phi(ge^t,u,y)-y^{-u})}{j_N(g,e^t)}dt.
    \end{split}
\end{equation}
\end{proof}
By the integral representation of the Lerch transcendent in Lemma \ref{lm:IRLT},
we have double integral representations of the Arakawa--Kaneko zeta-functions.
\begin{lemma}
\label{lm:DIR}
  \begin{align}
    \label{eq:IRXD}
    \xi_D(u,s;y,w;g)&=\frac{1}{\Gamma(s)\Gamma(u)}
    \int_0^{\infty}\int_0^{\infty}
    \frac{t^{s-1}x^{u-1}e^{-wt}e^{-yx}}{j_D(g,e^t)}
    \frac{1}{1-(ge^t) e^{-x}}
    dtdx,
\\
    \label{eq:IRXN}
    \xi_N(u,s;y,w;g)&=\frac{1}{\Gamma(s)\Gamma(u)}
    \int_0^{\infty}\int_0^{\infty}
    \frac{t^{s-1}x^{u-1}e^{-wt}e^{-yx}}{j_N(g,e^t)}
    \frac{(ge^t) e^{-x}}{1-(ge^t) e^{-x}}
    dtdx.
  \end{align}
\end{lemma}

To determine the domain of $\xi_D(u,s;y,w;g)$,
we need to study when the integral \eqref{eq:IRXD} is convergent.
Here we give a sufficient condition below. It should be noted that
generally the domain is wider and is dependent on $g$. 
To describe the domain, we define the following constants.

\begin{definition}\label{def-nu-mu}
Consider the two circular arcs $g([1,+\infty])$ and $[1,+\infty]$ on the Riemann sphere $\widehat{\mathbb{C}}$.
Then
for $T_0,X_0\in\{1,\infty\}$, we fix $\mu_{T_0},\nu_{X_0}\geq 0$ as follows.
For $gT_0\notin V(g)$ (resp.~$X_0\notin V(g)$), we set $\mu_{T_0}=0$ (resp.~$\nu_{X_0}=0$).
Further for a pair $(T_0,X_0)$ such that $gT_0=X_0\in V(g)$, we set
\begin{equation}
    \mu_{T_0}+\nu_{X_0}=
  \begin{cases}
    1\qquad&\text{if $gT_0=X_0$ is not a cusp},\\
    2\qquad&\text{if $gT_0=X_0$ is a cusp}.
  \end{cases}
\end{equation}
\end{definition}

\begin{lemma}
\label{lm:est}
\begin{enumerate}
\item There exists $M>0$ such that
for all $(t,x)\in(0,+\infty)^2$,
\begin{equation}
\label{eq:est1}
  \Bigl|\frac{1}{j_D(g,e^t)}
  \frac{1}{1-(ge^t) e^{-x}}\Bigr|
  \leq
  M
  t^{-\mu_1}x^{-\nu_1}e^{(\mu_\infty-1) t}
  e^{\nu_\infty x}(t+1)^{\mu_1}(x+1)^{\nu_1}.
\end{equation}

\item Let $Z$ be a neighborhood of $\{(\log T_0,\log X_0)~|~X_0\in V(g),T_0=g^{-1}X_0\}$ in $\mathbb{C}^2$. Then
there exist $M>0$ and $\epsilon>0$ such that
for all $(t,x)\in W_{0,\epsilon}^2\setminus Z$,
\begin{equation}
\label{eq:est3}
  \Bigl|\frac{1}{j_D(g,e^t)}
  \frac{1}{1-(ge^t) e^{-x}}\Bigr|
  \leq
  M e^{-\Re t}.
\end{equation}
\item If $g1\neq1$, then for any sufficiently large $R>0$,
there exist $M>0$ and $\epsilon>0$ such that
for all $(t,x)\in W_{0,\epsilon,R}^2$,
\begin{equation}
\label{eq:est4}
  \Bigl|\frac{1}{j_D(g,e^t)}
  \frac{1}{1-(ge^t) e^{-x}}\Bigr|
  \leq
  M.
\end{equation}
\item
If $g1=\infty$, then
there exists $\epsilon>0$ such that
\begin{equation}
\label{eq:dist1}
|t|>\epsilon e^{-x}
\end{equation}
for any 
pair $(t,x)$ satisfying
$ge^t=e^x$ in
a sufficiently small neighborhood
of $(0,+\infty)$ in $\mathbb{C}\times\mathbb{R}$.
\item
If $g\infty=1$, then
there exists $\epsilon>0$ such that
\begin{equation}
\label{eq:dist2}
|x|>\epsilon e^{-t}
\end{equation}
for any 
pair $(t,x)$ satisfying
$ge^t=e^x$ in
a sufficiently small neighborhood
of $(+\infty,0)$ in $\mathbb{R}\times\mathbb{C}$.
\end{enumerate}
\end{lemma}
\begin{proof}
Let $Z$ be a neighborhood of $\{(\log T_0,\log X_0)~|~X_0\in V(g),T_0=g^{-1}X_0\}$ in $\mathbb{C}^2$.
If $V(g)\neq\emptyset$, then for each $X_0\in V(g)$,
consider a sufficiently small neighborhood $N'(X_0)$ of $(T_0,X_0)$ in $(1,+\infty)^2$ such that $J'(X_0)=\{(\log T,\log X)~|~(T,X)\in N'(X_0)\}\subset Z$.
By Lemma \ref{lm:boundV}, 
there exists $M>0$ such that
\begin{equation}
  \Bigl|\frac{1}{j_D(g,e^t)}\frac{1}{1-(ge^t)e^{-x}}\Bigr|\leq 
  Me^{-t}
  \Bigl|\dfrac{e^t}{\Tilde{T_0}e^t-1}\Bigr|^{\mu_{T_0}}
\Bigl|\dfrac{e^x}{\Tilde{X_0}e^x-1}\Bigr|^{\nu_{X_0}}
\end{equation}
for all $(t,x)\in J'(X_0)$. 

Let $N$ be a sufficiently small neighborhood of
$\{(T_0,X_0)~|~X_0\in V(g),T_0=g^{-1}X_0\}$ in $\widehat{\mathbb{C}}^2$
such that $N\cap(1,+\infty)^2$ is contained in
the union of the neighborhoods $N'(X_0)$ 
taken in the previous paragraph for each $X_0\in V(g)$.
By Lemma \ref{lm:boundNV}, 
there exist $\epsilon>0$ and $M>0$ such that 
\begin{equation}
\label{eq:boundW}
  \Bigl|\frac{1}{j_D(g,e^t)}\frac{1}{1-(ge^t)e^{-x}}\Bigr|\leq 
  Me^{-\Re t}
\end{equation}
for all $(t,x)\in J$,
where
 $I=W_{1,\epsilon}^2\setminus N$ 
and $J=\{(\log T,\log X)~|~(T,X)\in I\}$.
 Let $\epsilon'>0$ be sufficiently small such that
$e^z\in W_{1,\epsilon}$ for all $z\in W_{0,\epsilon'}$.
This implies \eqref{eq:est3}.
In particular,
if $g1\neq1$, then $W_{0,\epsilon',R}^2\subset J$ for any sufficiently large $R>0$. Thus
\eqref{eq:boundW} implies \eqref{eq:est4}.

Since 
for all $z> 0$,
\begin{equation}
  \begin{split}
    1\leq \Bigl|\frac{e^z}{\Tilde{Z_0}e^z-1}\Bigr|
    &\leq 
    \begin{cases}
       \dfrac{z+1}{z}
      \qquad&(Z_0=1),\\    
      e^z\qquad&(Z_0=\infty),
    \end{cases}
  \end{split}
\end{equation}
we have
\begin{align}
1,\  \Bigl|\dfrac{e^t}{\Tilde{T_0}e^t-1}\Bigr|^{\mu_{T_0}}&\leq
t^{-\mu_1}e^{\mu_\infty t}(t+1)^{\mu_1},\\
1,\  \Bigl|\dfrac{e^x}{\Tilde{X_0}e^x-1}\Bigr|^{\nu_{X_0}}&\leq
x^{-\nu_1}e^{\nu_\infty x}(x+1)^{\nu_1}
\end{align}
and hence for all $(t,x)\in(0,+\infty)^2$,
\begin{equation}
  \begin{split}
    \Bigl|\frac{1}{j_D(g,e^t)}
  \frac{1}{1-(ge^t) e^{-x}}\Bigr|
  &\leq
  M'
 t^{-\mu_1}x^{-\nu_1}e^{(\mu_\infty-1) t}
e^{\nu_\infty x}(t+1)^{\mu_1}(x+1)^{\nu_1}
\end{split}
\end{equation}
for some $M'>0$, which implies \eqref{eq:est1}.



Inequalities
\eqref{eq:dist1} and
\eqref{eq:dist2}
follow from 
\eqref{eq:dist0}.
\end{proof}

\begin{theorem}
\label{thm:domain1}
For 
$\Re u>\nu_1,
\Re s>\mu_1,
\Re y>\nu_\infty,
\Re w>\mu_\infty-1$,
$\xi_D(u,s;y,w;g)$ is defined and analytic in $u,s,y,w$.
\end{theorem}
\begin{proof}
By Lemma \ref{lm:est},
\begin{equation}
  \begin{split}
    &\int_0^\infty\int_0^\infty 
  \Bigl|\frac{t^{s-1}x^{u-1}e^{-wt}e^{-yx}}{j_D(g,e^t)}
  \frac{1}{1-(ge^t) e^{-x}}\Bigr|
  dt dx
  \\
&\qquad
  \leq M
  \int_0^\infty\int_0^\infty 
  t^{\Re s-1-\mu_1}x^{\Re u-1-\nu_1}e^{(\mu_\infty-1-\Re w) t}
\\
&\qquad\qquad
 \times 
  e^{(\nu_\infty-\Re y) x}(t+1)^{\mu_1}(x+1)^{\nu_1}
  dtdx<\infty.
\end{split}
\end{equation}
The analyticity in $u,s,y,w$ follows from the Morera theorem and the Fubini theorem.
\end{proof}

\begin{example}\label{Exam-3-1}
Let $g_{\eta}:=
\begin{pmatrix} -1 & 1 \\ 0 & 1
\end{pmatrix}$. 
We can see that $g_\eta^{-1}=g_\eta$ and ${\rm det}\,g=-1$ which are important properties. 
For $g=g_\eta$, we have $gT=1-T$, namely, $g1=0$, $g\infty=\infty$ and 
\begin{equation*}
  g([1,+\infty])\cap[1,+\infty]=\{\infty\}= V(g).
\end{equation*}
Hence, by Definition \ref{def-nu-mu}, we obtain $\mu_1=0$ and $\nu_1=0$. Since $\infty$ is not a cusp, we have $\mu_\infty,\,\nu_\infty\in [0,1]$ satisfying $\mu_\infty+\nu_\infty=1$. Therefore $\xi_D(u,s;y,w;g_\eta)$ is defined for 
$  \Re u>0,
\Re s>0,
  \Re y>\nu_\infty,
  \Re w>\mu_\infty-1$,
 where $\mu_\infty,\,\nu_\infty\in [0,1]$ with $\mu_\infty+\nu_\infty=1$. 
We see that
$$\Li_u(g\,e^t)=\Li_u(1-e^t),\quad j_D(g,e^t)=1,\quad j_N(g,e^t)=1-e^t.$$
Hence, noting \eqref{Phi-Li} and \eqref{xi-D-N}, we define 
\begin{equation}
\eta(u;s)=\xi_N(u,s;0,0;g_\eta)=\xi_D(u,s;1,0;g_\eta),
\label{eta-xi}
\end{equation}
which was already considered by Yamamoto \cite{Yamamoto}. 

Let $g_{\xi}:=
\begin{pmatrix} 1 & -1 \\ 1 & 0
\end{pmatrix}$. Then ${\rm det}\,g_\xi=1$. 
For $g=g_\xi$, we have $gT=1-T^{-1}$, namely, $g1=0$, $g\infty=1$ and 
\begin{equation*}
  g([1,+\infty])\cap[1,+\infty]=\{1\}= V(g).
\end{equation*}
Hence we obtain $\mu_1=0$ and $\nu_{\infty}=0$. Since $1$ is not a cusp, we have $\mu_\infty,\,\nu_1\in [0,1]$ satisfying $\mu_\infty+\nu_1=1$. Therefore $\xi_D(u,s;y,w;g_\xi)$ is defined for 
$  \Re u>\nu_1,
\Re s>0,
  \Re y>0,
  \Re w>\mu_\infty-1$,
 where $\mu_\infty,\,\nu_1\in [0,1]$ with $\mu_\infty+\nu_1=1$. 
We have 
$$\Li_u(g\,e^t)=\Li_u(1-e^{-t}),\quad j_D(g,e^t)=e^t,\quad j_N(g,e^t)=e^t-1.$$
Hence, noting \eqref{Phi-Li} and \eqref{xi-D-N}, we define 
\begin{equation}
\xi(u;s;w)=\xi_N(u,s;0,w-1;g_\xi)=\xi_D(u,s;1,w-1;g_\xi)
\end{equation}
and, in particular, 
\begin{equation}
\xi(u;s)=\xi_N(u,s;0,0;g_\xi)=\xi_D(u,s;1,0;g_\xi), \label{xi-def-ND}
\end{equation}
which is a generalization of \eqref{1-9}. 
\end{example}

\section{Relations among Arakawa--Kaneko zeta-functions}\label{sec-4}

In this section, we give two types of functional relation formulas for $\xi_D$ and $\xi_N$ (see Theorems \ref{thm:trans1} and \ref{thm:trans2}). 
We will see that
these give functional relations which interpolate the well-known relations among poly-Bernoulli numbers in Section \ref{sec-6}.

For 
$g=
\begin{pmatrix}
  a & b \\ c & d
\end{pmatrix}$,
we put
$j_N(g,D_u)=aD_u+b$ for the difference operator $D_u$ and so on.
\begin{theorem}[Difference relations]
\label{thm:trans1}
For 
$g=
\begin{pmatrix}
  a & b \\ c & d
\end{pmatrix}$, 
we have
  \begin{equation}\label{eq:rem-4-2}
    j_N(g,D_w^{-1})\xi_{N}(u,s;y,w;g)=
    j_D(g,D_w^{-1})\xi_{D}(u,s;y,w;g)-y^{-u}w^{-s},
  \end{equation}
namely,
  \begin{equation}
    \begin{split}
    &a\xi_{N}(u,s;y,w-1;g)+b\xi_{N}(u,s;y,w;g)\\
    \bigl(=&a\xi_{D}(u,s;y+1,w-1;g)+b\xi_{D}(u,s;y+1,w;g)\bigr)\\
    &\qquad\qquad=c\xi_{D}(u,s;y,w-1;g)+d\xi_{D}(u,s;y,w;g)-y^{-u}w^{-s}.
  \end{split}
\end{equation}
\end{theorem}
\begin{proof}
The assertion follows
from the integral representations \eqref{eq:IR1} and \eqref{eq:IR2} with
  \begin{equation}
\frac{aT+b}{j_N(g,T)}-\frac{cT+d}{j_D(g,T)}
=1-1=0
  \end{equation}
and
\begin{equation}
\frac{1}{\Gamma(s)}\int_0^{\infty} t^{s-1}e^{-wt}y^{-u}dt
=y^{-u}w^{-s}.
\end{equation}
\end{proof}

For $g\in\GL_2(\mathbb{C})$ and 
indeterminates $X,T$, we define
\begin{align}
  F_D&=\frac{1}{1-(gT)X^{-1}},&
  G_D&=\frac{1}{1-(g^{-1}X)T^{-1}},
\\
  F_N&=\frac{(gT)X^{-1}}{1-(gT)X^{-1}},&
  G_N&=\frac{(g^{-1}X)T^{-1}}{1-(g^{-1}X)T^{-1}}.
\end{align}
From Lemma \ref{lm:DIR}, we see that these come from the integrands of the double integral representations. We have the key relations, which are the core of duality relations.
\begin{lemma}
\label{lm:FG}
\begin{align}
  \frac{T}{j_D(g,T)}F_D&=-\frac{1}{\det g}\frac{X}{j_D(g^{-1},X)}G_D,\\
  \frac{1}{j_N(g,T)}F_N&=-\frac{1}{\det g}\frac{1}{j_N(g^{-1},X)}G_N,\\
  \frac{1}{j_D(g,T)}F_D&=-\frac{1}{\det g}\frac{X}{j_N(g^{-1},X)}G_N.
\end{align}
\end{lemma}
\begin{proof}  
The first equation follows from
\begin{equation}
  \begin{split}
    F_D
    &=\frac{1}{1-(gT)X^{-1}}
    \\
    &=\frac{j_D(g,T) X}{(cX-a)T-(-dX+b)}
    \\
    &=-\frac{j_D(g,T) T^{-1}}{1-(g^{-1}X)T^{-1}}\frac{X}{-cX+a}
    \\
    &=-\frac{1}{\det g}\frac{j_D(g,T) T^{-1}}{1-(g^{-1}X)T^{-1}}
    \frac{X}{j_D(g^{-1},X)}
    \\
    &=-\frac{1}{\det g}\frac{j_D(g,T) T^{-1}X}{j_D(g^{-1},X)}G_D,
  \end{split}
\end{equation}
and the third, from 
\begin{equation}
  \begin{split}
    F_D&=
    -\frac{1}{\det g}\frac{j_D(g,T) T^{-1}X}{j_D(g^{-1},X)}
    \frac{j_D(g^{-1},X)}{j_N(g^{-1},X)T^{-1}}G_N
    \\
    &=
    -\frac{1}{\det g}\frac{j_D(g,T) X}{j_N(g^{-1},X)}G_N,
  \end{split}
\end{equation}
and finally the second, from
\begin{equation}
  \begin{split}
    F_N&=\frac{j_N(g,T)X^{-1}}{j_D(g,T)}F_D
    \\
    &=-\frac{1}{\det g}\frac{j_N(g,T)X^{-1}}{j_D(g,T)}\frac{j_D(g,T) X}{j_N(g^{-1},X)}G_N
    \\
    &=-\frac{1}{\det g}\frac{j_N(g,T)}{j_N(g^{-1},X)}G_N.
  \end{split}
\end{equation}
\end{proof}

There are three types of duality relations, namely,
ascending-ascending, descending-descending, and ascending-descending
types.

\begin{theorem}[Duality relations]
\label{thm:trans2}
For $n\in \mathbb{Z}_{\geq 0}$, 
\begin{multline}\label{dual-1}
  j_D(g^{-1},D_y^{-1})^{n}D_w^{-n-1}\Bigl(\prod_{k=1}^n(D_u^{-1}-y+k)\Bigr)
  \xi_{D}(u,s;y,w;g)\\
  \begin{aligned}
    &=\Bigl(\frac{-1}{\det g}\Bigr)^{n+1}
    j_D(g,D_w^{-1})^{n}D_y^{-n-1}\Bigl(\prod_{k=1}^n(D_s^{-1}-w+k)\Bigr)
    \xi_{D}(s,u;w,y;g^{-1}),
  \end{aligned}
\end{multline}
\begin{multline}\label{dual-2}
  j_N(g^{-1},D_y^{-1})^{n}\Bigl(\prod_{k=1}^n(D_u^{-1}-y-k)\Bigr)
  \xi_{N}(u,s;y,w;g)\\
  \begin{aligned}
    &=\Bigl(\frac{-1}{\det g}\Bigr)^{n+1}
    j_N(g,D_w^{-1})^{n}\Bigl(\prod_{k=1}^n(D_s^{-1}-w-k)\Bigr)
  \xi_{N}(s,u;w,y;g^{-1})
\end{aligned}
\end{multline}
and
\begin{multline}\label{dual-3}
  j_N(g^{-1},D_y^{-1})^{n}\Bigl(\prod_{k=1}^n(D_u^{-1}-y+k)\Bigr)
  \xi_{D}(u,s;y,w;g)\\
  \begin{aligned}
    &=\Bigl(\frac{-1}{\det g}\Bigr)^{n+1}
    j_D(g,D_w^{-1})^{n}D_y^{-n-1}\Bigl(\prod_{k=1}^n(D_s^{-1}-w-k)\Bigr)
    \xi_{N}(s,u;w,y;g^{-1}).
  \end{aligned}
\end{multline}
\end{theorem}
\begin{proof}
From
Lemmas \ref{lm:rd}, \ref{lm:DS} and \ref{lm:DN}, we have
  \begin{multline}
    \frac{1}{n!}\Bigl(\prod_{k=1}^n(D_u^{-1}-y+k)\Bigr)
    \xi_{D}(u,s;y,w;g)
\\
\begin{aligned}
  &=
\frac{1}{\Gamma(s)\Gamma(u)}
    \int_0^{\infty}\int_0^{\infty}
    \frac{t^{s-1}x^{u-1}e^{-wt}e^{-yx}}{j_D(g,e^t)}
    \frac{1}{(1-(ge^t) e^{-x})^{n+1}}
    dtdx,
  \end{aligned}
\end{multline}
and
\begin{multline}
  \frac{1}{n!}\Bigl(\prod_{k=1}^n(D_u^{-1}-y-k)\Bigr)
    \xi_N(u,s;y,w;g)
    \\
    \begin{aligned}
      &=\frac{1}{\Gamma(s)\Gamma(u)}
      \int_0^{\infty}\int_0^{\infty}
      \frac{t^{s-1}x^{u-1}e^{-wt}e^{-yx}}{j_N(g,e^t)}
      \Bigl(\frac{(ge^t) e^{-x}}{1-(ge^t) e^{-x}}\Bigr)^{n+1}
      dtdx.
    \end{aligned}
  \end{multline}
Lemma \ref{lm:FG} implies
\begin{align}
  j_D(g^{-1},X)^{n}\frac{T^{n+1}}{j_D(g,T)}F_D^{n+1}&=
\Bigl(\frac{-1}{\det g}\Bigr)^{n+1}
j_D(g,T)^{n}\frac{X^{n+1}}{j_D(g^{-1},X)}G_D^{n+1},\\
j_N(g^{-1},X)^{n}\frac{1}{j_N(g,T)}F_N^{n+1}&=
\Bigl(\frac{-1}{\det g}\Bigr)^{n+1}
j_N(g,T)^{n}\frac{1}{j_N(g^{-1},X)}G_N^{n+1},\\
j_N(g^{-1},X)^{n}  \frac{1}{j_D(g,T)}F_D^{n+1}&=
\Bigl(\frac{-1}{\det g}\Bigr)^{n+1}
j_D(g,T)^{n}\frac{X^{n+1}}{j_N(g^{-1},X)}G_N^{n+1}.
\end{align}
By noting that 
\begin{equation}
  j_D(g^{-1},D_y^{-1})^{n}D_w^{-n-1}(t^{s-1}x^{u-1}e^{-wt}e^{-yx})
= j_D(g^{-1},e^x)^{n}e^{(n+1)t}(t^{s-1}x^{u-1}e^{-wt}e^{-yx})
\end{equation}
and so on,
we obtain the result.
\end{proof}

The $n=0$ case reduces to the following.
\begin{corollary}
\label{thm:trans3}
  \begin{align}
    \label{eq:Tr2}
    \xi_{D}(u,s;y,w-1;g)&=-\frac{1}{\det g}\xi_{D}(s,u;w,y-1;g^{-1}),\\
    \label{eq:Tr1}
    \xi_{N}(u,s;y,w;g)&=-\frac{1}{\det g}\xi_{N}(s,u;w,y;g^{-1}),\\
    \label{eq:Tr3}
    \xi_{D}(u,s;y,w;g)&=-\frac{1}{\det g}\xi_{N}(s,u;w,y-1;g^{-1}),
  \end{align}
which are essentially the same formulas.
\end{corollary}

\begin{example}\label{Exam-4-6}\ \ 
As for $\eta(u;s)=\xi_D(u,s;1,0;g_\eta)$ defined in Example \ref{Exam-3-1}, noting $g_\eta^{-1}=g_\eta$, we see that \eqref{eq:Tr2} with $(y,w)=(1,1)$ implies Yamamoto's result $\eta(u;s)=\eta(s;u)$ in \eqref{Yamamoto-F}, which interpolates \eqref{1-5} (for the values of $\eta(u;s)$ at nonpositive integers, see \eqref{eta-xi-vals}). 
We will further introduce several duality relations for $\xi_D(u,s;y,w;g)$ in Section \ref{sec-5} (see Examples \ref{Exam-5-6} and \ref{Exam-5-6-3}). 
\end{example}

\begin{remark}
  $\xi(u,s;y,w;g)$ can be slightly generalized with two elements $g,h\in\GL_2(\mathbb{C})$ and two appropriate paths $I,J$ which starts at $0$ and goes to $+\infty$ as
\begin{equation}
    \xi(u,s;y,w;h,g;I,J)
    =\frac{1}{\Gamma(s)\Gamma(u)}
    \int_{J} dt\int_{I} dx
    \frac{t^{s-1}x^{u-1}e^{-wt}e^{-yx}}{j_D(g,e^t)j_D(h,e^x)}
    \frac{1}{he^x-ge^t}.
\end{equation}
Since variables are treated completely symmetrically,
it is easy to see that the trivial symmetry
\begin{equation}
    \xi(u,s;y,w;h,g;I,J)
    =-\xi(s,u;w,y;g,h;J,I)
\end{equation}
holds. Moreover we can show that
\begin{equation}
    \xi(u,s;y-1,w;h,g;(0,+\infty),(0,+\infty))
    =\frac{1}{\det h}\xi_D(u,s;y,w;h^{-1}g),
\end{equation}
which implies \eqref{eq:Tr2}.

From the above we see that the pair $(g,h)$ does not give rise to a generalization,
while two pathes $I,J$ are essential because by this modification, it is possible to avoid cusps and to define $\xi_D$ 
for any element $g\in\GL_2(\mathbb{C})$
without the restriction
\eqref{eq:def_cond}.
\end{remark}

\section{Analytic continuation}\label{sec-5}

We give integral representations with Hankel contours to enlarge the
domain of $\xi_D(u,s;y,w;g)$. In the following, $H_{\epsilon,R}$ denotes the Hankel
contour, which consists of a path from $R$ to $\epsilon$ on the
real axis, around the origin counter clockwise with radius $\epsilon$,
and back to $R$, where $R\in(0,+\infty]$ and $\epsilon$ is an arbitrarily small
positive number. We abbreviate $H=H_{\epsilon,+\infty}$.

In the following proofs,
since the analyticities follow from the Morera theorem and the Fubini theorem,
we omit them.

\begin{lemma}
Let $k\in\mathbb{Z}_{\geq0}$. Then
$\Phi(z,u,y)$ has the integral representation
\begin{equation}\label{Phi-repre}
    \Phi(z,u,y)
    =
    \frac{1}{\Gamma(u)(e^{2\pi iu}-1)}
    \int_H x^{u-1}e^{-(y+k)x}\frac{z^k}{1-z e^{-x}}dx
    +\sum_{n=0}^{k-1}\frac{z^n}{(y+n)^{u}}.
\end{equation}
This expression gives
the analytic continuation of
$\Phi(z,u,y)$ and is valid for
$z\in\mathbb{C}\setminus [1,+\infty)$ or $z=1$, $u\in\mathbb{C}$ and 
$y\in\mathbb{C}\setminus (-\infty,0]$ with
$\Re y>-k$
except for appropriate branch cuts.
Therefore $\Phi(z,u,y)$ is analytically continued in
$z\in\mathbb{C}\setminus [1,+\infty)$ or $z=1$, $u\in\mathbb{C}$ and 
$y\in\mathbb{C}\setminus (-\infty,0]$
except for appropriate branch cuts.
\end{lemma}
\begin{proof}
  By \eqref{eq:FELT1} and Lemma \ref{lm:IRLT},
we have
\begin{equation}
    \Phi(z,u,y)
    =
    \frac{1}{\Gamma(u)}
    \int_0^{\infty} x^{u-1}e^{-(y+k)x}\frac{z^k}{1-z e^{-x}}dx
    +\sum_{n=0}^{k-1}\frac{z^n}{(y+n)^{u}},
\end{equation}
which gives the integral representation with the Hankel contour.
\end{proof}

We study integral representations of $\xi_D(u,s;y,w;g)$
with Hankel contours by considering slightly general forms given in Lemma \ref{lm:DIR}, namely, 
for $k\in\mathbb{Z}_{\geq0}$,
  \begin{equation}
    \label{eq:xiD_k}
    \begin{split}
      \xi_D(u,s;y,w;g)&=\frac{1}{\Gamma(s)\Gamma(u)}
      \int_0^{\infty}\int_0^{\infty}
      \frac{t^{s-1}x^{u-1}e^{-wt}e^{-(y+k)x}}{j_D(g,e^t)}
      \frac{(ge^t)^k}{1-(ge^t) e^{-x}}
    dtdx
    \\
    &\qquad+
      \frac{1}{\Gamma(s)}
      \sum_{n=0}^{k-1}\frac{1}{(y+n)^{u}}\int_0^{\infty} t^{s-1}e^{-wt}\frac{(ge^t)^n}{j_D(g,e^t)}dt,
  \end{split}
\end{equation}
by \eqref{eq:FELT2} and \eqref{eq:IR1}.
We denote the first term and the second term 
by $\xi_{1,k}(u,s;y,w;g)$ and $\xi_{2,k}(u,s;y,w;g)$ respectively so that
\begin{equation}
  \xi_D(u,s;y,w;g)=\xi_{1,k}(u,s;y,w;g)+\xi_{2,k}(u,s;y,w;g).
\end{equation}

First we give the explicit form of $\xi_{2,k}(u,s;y,w;g)$,
which gives its analytic continuation.
\begin{lemma}
\label{lm:xi_secondp}
Let $k\in\mathbb{Z}_{\geq0}$.
\begin{equation}
\label{eq:xi_secondp}
  \begin{split}
    &\xi_{2,k}(u,s;y,w;g)=
    \sum_{n=0}^{k-1}
      \frac{1}{(y+n)^{u}}
      j_N(g,D_w^{-1})^n
    \\
    &\qquad\times \begin{cases}
      \displaystyle
      \frac{1}{d^{n+1}}
      \frac{1}{w^s}
      \qquad&(g\infty=\infty),
      \\
      \displaystyle
      \frac{1}{c^{n+1}}
      \frac{1}{(w+n+1)^s}
      \qquad&(g0=\infty),
      \\
      \displaystyle
      \frac{1}{c^{n+1}}
      \frac{1}{n!}D_w^{n+1}
      \Bigl(\prod_{j=1}^n(D_s^{-1}-w+j)\Bigr)
      \Phi(-d/c,s,w)\qquad&(\text{otherwise}),
    \end{cases}
  \end{split}
\end{equation}
which gives the analytic continuation to the whole space in $u,s,y,w$ except for appropriate branch cuts.
\end{lemma}

\begin{proof}
If $g\infty=\infty$, then $c=0$ and
\begin{equation}
  \begin{split}
    \int_0^{\infty} t^{s-1}e^{-wt}\frac{(ge^t)^n}{j_D(g,e^t)}dt
    &= \frac{1}{d^{n+1}}j_N(g,D_w^{-1})^n
    \int_0^{\infty} t^{s-1}e^{-wt}dt
    \\
    &=\frac{\Gamma(s)}{d^{n+1}}j_N(g,D_w^{-1})^n\frac{1}{w^s},
  \end{split}
\end{equation}
which implies \eqref{eq:xi_secondp} in this case.
If $g0=\infty$, then $d=0$ and
\begin{equation}
  \begin{split}
    \int_0^{\infty} t^{s-1}e^{-wt}\frac{(ge^t)^n}{j_D(g,e^t)}dt
    &= \frac{1}{c^{n+1}}j_N(g,D_w^{-1})^n
    \int_0^{\infty} t^{s-1}e^{-(w+n+1)t}dt
    \\
    &=\frac{\Gamma(s)}{c^{n+1}}j_N(g,D_w^{-1})^n\frac{1}{(w+n+1)^s},
  \end{split}
\end{equation}
which implies \eqref{eq:xi_secondp} in this case.
If $g0,g\infty\neq \infty$, then $c,d\neq 0$ and
\begin{equation}
  \begin{split}
    \int_0^{\infty} t^{s-1}e^{-wt}\frac{(ge^t)^n}{j_D(g,e^t)}dt
    &=\frac{1}{c^{n+1}}j_N(g,D_w^{-1})^nD_w^{n+1}
    \int_0^{\infty} t^{s-1}e^{-wt}\frac{1}{(1-(-d/c)e^{-t})^{n+1}}dt
    \\
    &=
    \frac{\Gamma(s)}{c^{n+1}}j_N(g,D_w^{-1})^nD_w^{n+1}
    \frac{1}{n!}
    \Bigl(\prod_{j=1}^n(D_s^{-1}-w+j)\Bigr)
    \Phi(-d/c,s,w),
  \end{split}
\end{equation}
by Lemma \ref{lm:DS}. 
If $-d/c=T$ with $1<T<\infty$, then $cT+d=0$, which implies $\{\infty\}\in g((1,+\infty))$ and
contradicts to the assumption \eqref{eq:def_cond}
and hence $-d/c\in\mathbb{C}\setminus(1,+\infty)$.
Hence we obtain \eqref{eq:xi_secondp} in this case.
\end{proof}

\begin{theorem}
\label{thm:domain2}
Let $k\in\mathbb{Z}_{\geq0}$.
Assume $g1\neq1$.
Then
we have
\begin{multline}
\label{eq:ext3k}
  \xi_D(u,s;y,w;g)\\
  \begin{aligned}
    &=\frac{1}{\Gamma(s)\Gamma(u)(e^{2\pi is}-1)(e^{2\pi iu}-1)}
  \int_{H_{\epsilon,1}} dx \int_{H_{\epsilon,1}} dt
  \frac{t^{s-1}x^{u-1}e^{-wt}e^{-(y+k)x}}{j_D(g,e^t)}
  \frac{(ge^t)^k}{1-(ge^t) e^{-x}}
  \\
  &\qquad+
  \frac{1}{\Gamma(s)\Gamma(u)(e^{2\pi is}-1)}
  \int_1^\infty dx \int_{H_{\epsilon e^{-x},1}} dt
  \frac{t^{s-1}x^{u-1}e^{-wt}e^{-(y+k)x}}{j_D(g,e^t)}
  \frac{(ge^t)^k}{1-(ge^t) e^{-x}}
  \\
  &\qquad+
  \frac{1}{\Gamma(s)\Gamma(u)(e^{2\pi iu}-1)}
  \int_1^\infty dt \int_{H_{\epsilon e^{-t},1}} dx
  \frac{t^{s-1}x^{u-1}e^{-wt}e^{-(y+k)x}}{j_D(g,e^t)}
  \frac{(ge^t)^k}{1-(ge^t) e^{-x}}
  \\
  &\qquad+
  \frac{1}{\Gamma(s)\Gamma(u)}
  \int_1^\infty dt \int_1^\infty dx
  \frac{t^{s-1}x^{u-1}e^{-wt}e^{-(y+k)x}}{j_D(g,e^t)}
  \frac{(ge^t)^k}{1-(ge^t) e^{-x}}
  \\
  &\qquad+\xi_{2,k}(u,s;y,w;g),
\end{aligned}
\end{multline}
which except for the branch cuts due to $\xi_{2,k}(u,s;y,w;g)$,
gives the analytic continuation for
$u, s\in\mathbb{C}$, $\Re y>\nu_\infty-k$, $\Re w>\mu_\infty-1+k(\delta_{g\infty,\infty}-\delta_{g\infty,0})$, and
the continuous extension for
  $\Re y=\nu_\infty-k$ when $\Re u<0$ 
and $\Re w=\mu_\infty-1+k(\delta_{g\infty,\infty}-\delta_{g\infty,0})$ when $\Re s<0$.
\end{theorem}
\begin{proof}
  There exists $M>0$ such that for all sufficiently large $t>R'$, 
\begin{equation}
  \begin{split}
    |ge^t| &\leq
    \begin{cases}
      Me^{-t}\qquad&(g\infty=0),\\
      Me^t\qquad&(g\infty=\infty),\\
      M\qquad&(\text{otherwise})
    \end{cases}\\
    &=Me^{(\delta_{g\infty,\infty}-\delta_{g\infty,0})t}
  \end{split}
\end{equation}
and for all sufficiently small $|t|<\epsilon'$
\begin{equation}
  \begin{split}
    |ge^t| &\leq
   \begin{cases}
     M|t|\qquad&(g1=0),\\
     M|t|^{-1}\qquad&(g1=\infty),\\
     M\qquad&(\text{otherwise})
   \end{cases}\\
   &=M|t|^{(\delta_{g1,0}-\delta_{g1,\infty})}.
 \end{split}
\end{equation}
Assume $g1\neq1$. 
By Lemma \ref{lm:est},
\begin{multline}
  \Bigl|\frac{t^{s-1}x^{u-1}e^{-wt}e^{-yx}}{j_D(g,e^t)}
  \frac{1}{1-(ge^t) e^{-x}}\Bigr| 
  \\
  \leq
  M'
  \begin{cases}
    e^{(\mu_\infty-1-\Re w)t}
    e^{(\nu_\infty -\Re y)x}\qquad&((t,x)\in(\epsilon',+\infty)^2),\\
    |t|^{\Re s-1}
    |x|^{\Re u-1}\qquad&((t,x)\in W_{0,\epsilon',R'}^2).
\end{cases}
\end{multline}
Thus for $\Re u>\nu_1, 
\Re s>\mu_1+k(\delta_{g1,\infty}-\delta_{g1,0}),
\Re y>\nu_\infty-k,
\Re w>\mu_\infty-1+k(\delta_{g\infty,\infty}-\delta_{g\infty,0})$,
we see that
\begin{equation}
  \int_{\epsilon'}^\infty
  \int_{\epsilon'}^\infty
  \frac{t^{s-1}x^{u-1}e^{-wt}e^{-(y+k)x}}{j_D(g,e^t)}
  \frac{(ge^t)^k}{1-(ge^t) e^{-x}} dtdx
\end{equation}
and
\begin{multline}
  \int_0^{R'}
  \int_0^{R'}
  \frac{t^{s-1}x^{u-1}e^{-wt}e^{-(y+k)x}}{j_D(g,e^t)}
  \frac{(ge^t)^k}{1-(ge^t) e^{-x}} dtdx
  \\
  =\frac{1}{(e^{2\pi is}-1)(e^{2\pi iu}-1)}
  \int_{H_{\epsilon,R'}}   \int_{H_{\epsilon,R'}} 
  \frac{t^{s-1}x^{u-1}e^{-wt}e^{-(y+k)x}}{j_D(g,e^t)}
  \frac{(ge^t)^k}{1-(ge^t) e^{-x}} dtdx
\end{multline}
are integrable. Let
\begin{equation}
  A=\int_{R'}^\infty dx 
  \int_0^{\epsilon'} dt
  \frac{t^{s-1}x^{u-1}e^{-wt}e^{-(y+k)x}}{j_D(g,e^t)}
  \frac{(ge^t)^k}{1-(ge^t) e^{-x}}.
\end{equation}
If $g1=\infty$, then by \eqref{eq:dist1}, the denominator does not vanish 
for $\epsilon e^{-x}\geq |t|$. Hence
\begin{equation}
\label{eq:HA}
  A=\frac{1}{e^{2\pi i s}-1}\int_{R'}^\infty dx 
  \int_{H_{\epsilon e^{-x},\epsilon'}} dt
  \frac{t^{s-1}x^{u-1}e^{-wt}e^{-(y+k)x}}{j_D(g,e^t)}
  \frac{(ge^t)^k}{1-(ge^t) e^{-x}}.
\end{equation}
If $g1\neq\infty$, then it is easier to see that the denominator does not vanish in the same region as the above, and \eqref{eq:HA} holds.

In the region $(0,\epsilon')\times (R',\infty)$, the same argument works well and we have the assertion by rearranging the regions.
\end{proof}

\begin{remark}
For $k\in\mathbb{Z}_{<0}$, we have similar results as in Lemma \ref{lm:xi_secondp} and Theorem \ref{thm:domain2} by use of \eqref{eq:FELT2}, though we omit the detail.
\end{remark}

In the case $k=0$, we obtain the following theorem.
\begin{theorem}
\label{thm:domain3}
  If $g1\notin\{1,\infty\}$, then we have
\begin{equation}
\label{eq:ext1}
  \xi_D(u,s;y,w;g)=\frac{1}{\Gamma(s)\Gamma(u)(e^{2\pi is}-1)}
  \int_H dt \int_0^{\infty}dx 
  \frac{t^{s-1}x^{u-1}e^{-wt}e^{-yx}}{j_D(g,e^t)}
  \frac{1}{1-(ge^t) e^{-x}},
\end{equation}
which gives the analytic continuation for
$\Re u>\nu_1$, $s\in\mathbb{C}$, $\Re y>\nu_\infty$, $\Re w>\mu_\infty-1$, and
the continuous extension for
  $\Re w=\mu_\infty-1$ when $\Re s<0$. 
%
%
%
%
\end{theorem}
\begin{proof}
If $g1\notin\{1,\infty\}$, 
then 
in the proof of Theorem \ref{thm:domain2}, 
the radius of the Hankel contours can be taken uniformly in $t$
while $x\in(0,+\infty)$. 
Thus patching contours, we have the assertion.
\end{proof}

When $u$, $s$ or both are nonpositive integers, 
 further analytic continuation is possible, which leads us to generalizations of the poly-Bernoulli polynomials. 
\begin{theorem}\label{AC-uyw}
  Assume $g1\neq1$. 
For $s=-m\in\mathbb{Z}_{\leq0}$, 
$\xi_D(u,s;y,w;g)$ is analytically continued
to
$u,y,w\in\mathbb{C}$ except for appropriate branch cuts
and we have the integral representation
\begin{multline}\label{Eq-xi-2}
  \xi_D(u,-m;y,w;g)\\
  \begin{aligned}
    &=\frac{(-1)^m m!}{2\pi i}\frac{1}{\Gamma(u)(e^{2\pi iu}-1)}
    \int_{H} dx
    \int_{|t|=\epsilon e^{{}-\Re x}} dt 
    \frac{t^{-m-1}x^{u-1}e^{-wt}e^{-(y+k)x}}{j_D(g,e^t)}
    \frac{(ge^t)^k}{1-(ge^t) e^{-x}}
    \\
    &\qquad+\xi_{2,k}(u,-m;y,w;g).
  \end{aligned}
\end{multline}

For $u=-m\in\mathbb{Z}_{\leq0}$, 
$\xi_D(u,s;y,w;g)$ is analytically continued
to
$s,y,w\in\mathbb{C}$ except for appropriate branch cuts.
\end{theorem}
\begin{proof}
If $s=-m\in\mathbb{Z}_{\leq0}$, then by Theorem \ref{thm:domain2}, we see that the Hankel contour in the first and the second terms of \eqref{eq:ext3k} with respect to $t$ reduces to a small circle around the origin 
and that the third and the fourth terms vanish.
Thus we obtain the integral representation.
The integral converges for any $w\in\mathbb{C}$. Since the analytic continuation is valid for $\Re y>\nu_\infty-k$ with arbitrary $k\in\mathbb{Z}_{\geq0}$, we have the first assertion.

The second assertion follows from Corollary \ref{thm:trans3}.
\end{proof}

\begin{example}\label{Exam-5-eta-xi}\ \ 
For $g=g_\eta,g_\xi$ in Example \ref{Exam-3-1}, we see that $g1\neq 1$. Hence, by Theorem \ref{thm:domain2}, we see that $\xi_D(u,s;y,w;g_\eta)$ is analytic for $u,s\in \mathbb{C}$, $\Re y>\nu_\infty$  and $\Re w>\mu_\infty-1$ with $\mu_\infty+\nu_\infty=1$. In the case when $(y,w)=(1,0)$, $\eta(u;s)=\xi_D(u,s;1,0;g_\eta)$ is analytic for $u,s\in \mathbb{C}$. 
Furthermore, when $(y,w)=(1,-1)$, 
we can define 
\begin{equation}
\widetilde{\xi}(u;s)=\xi_N(u,s;0,-1;g_\eta)=\xi_D(u,s;1,-1;g_\eta)\label{tildexi-xi}
\end{equation}
for $u,s\in \mathbb{C}$ with $\Re u<0$ and $\Re s<0$. In particular when $u=-k\in\mathbb{Z}_{\leq0}$, by Theorem \ref{AC-uyw}, we see that $\widetilde{\xi}(-k;s)$ can be analytically continued to $s\in\mathbb{C}$, which was already considered in \cite[Section 4]{Kaneko-Tsumura2015}. 

Also $\xi_D(u,s;y,w;g_\xi)$ is analytic for $u,s\in \mathbb{C}$, $\Re y>0$ and $\Re w>-1$. 
In particular, $\xi(u;s)=\xi_D(u,s;1,0;g_\xi)$ is analytic for $u,s\in \mathbb{C}$. 
\end{example}

\begin{example}\label{Exam-5-6}\ \ 
Consider $\widetilde{\xi}(u;s)=\xi_D(u,s;1,-1;g_\eta)$. By \eqref{dual-1} with $(n,y,w,g)=(1,1,1,g_\eta)$, we obtain 
\begin{equation}
\widetilde{\xi}(u-1,s)=\widetilde{\xi}(s-1,u), \label{dual-tilde-xi}
\end{equation}
which interpolates \eqref{1-5-2} (see Example \ref{Exam-6-1}). 
Note that from \eqref{eq:Tr3} with $(y,w)=(0,0)$, we have
$$\xi_D(u,s;0,0;g_\eta)=\xi_N(s,u;0,-1;g_\eta)=\xi_D(s,u;1,-1;g_\eta).$$
Therefore it follows from \eqref{dual-tilde-xi} that
\begin{equation}
\xi_D(u,s;0,0;g_\eta)=\xi_D(u-1,s+1;1,-1;g_\eta)=\widetilde{\xi}(u-1,s+1). \label{eq:Tr-xi}
\end{equation}

Let $g=g_{\xi}$ and $\xi(u;s):=\xi_D(u,s;1,0;g_\xi)$ in Example \ref{Exam-3-1}. Since $g_{\xi}^{-1}=
\begin{pmatrix} 0 & 1 \\ -1 & 1
\end{pmatrix},$ 
we have $g_{\xi}^{-1}T=1/(1-T)$ which satisfies \eqref{eq:def_cond}. 
Let $h=g_\xi^{-1}$. Then 
$hT=1/(1-T)$, namely, $h1=\infty$, $h\infty=0$ and 
\begin{equation*}
  h([1,+\infty])\cap[1,+\infty]=\{\infty \}= V(h).
\end{equation*}
Hence we obtain $\mu_\infty=0$ and $\nu_{1}=0$. 
Therefore, noting \eqref{eq:Tr2} with $(y,w,g)=(0,0,g_\xi)$, we can define 
\begin{align}
\check{\xi}(u;s):&=\xi_D(u,s;0,-1;g_\xi^{-1}) =-\xi_D(s,u;0,-1;g_\xi) 
\label{def-check-xi}
\end{align}
for $\Re u<0$ and $\Re s<0$. 
Setting $(n,y,w,g)=(1,1,0,g_\xi)$ in \eqref{dual-3} and noting \eqref{xi-D-N}, we obtain 
\begin{align*}
& \xi_{D}(u-1,s;1,0;g_\xi)=\xi_{D}(s-1,u;0,-1;g_\xi^{-1}).
\end{align*}
Therefore we see from \eqref{def-check-xi} that
\begin{align}
& \xi(u-1;s)=\check{\xi}(s-1;u), \label{xi-dual-2}
\end{align}
which also interpolates \eqref{1-5-2} (see Example \ref{Exam-6-gene}). 
The symbol $\check{\xi}$ is derived from this fact. From this relation, 
$\check{\xi}(s;u)$ is analytic for $u,s\in \mathbb{C}$. 
\end{example}

\begin{example}\label{Exam-5-6-3}
From \eqref{dual-1} with $n=1$, we obtain
\begin{multline*}
j_D(g^{-1},D_y^{-1})D_w^{-2}(D_u^{-1}-y+1)\xi_D(u,s;y,w;g)
\\
=
\Bigl(-\frac{1}{\det g}\Bigr)^2j_D(g,D_w^{-1})D_y^{-2}(D_s^{-1}-w+1)
\xi_D(s,u;w,y;g^{-1}).
\end{multline*}
Substituting \eqref{eq:Tr3} into the right-hand side and noting $j_D(g,D_w^{-1})w=wj_D(g,D_w^{-1})-cD_w^{-1}$, we have
\begin{multline}
\begin{aligned}
& j_D(g^{-1},D_y^{-1})D_w^{-2}(D_u^{-1}-y+1)\xi_D(u,s;y,w;g)
\\
  &=
\Bigl(-\frac{1}{\det g}\Bigr)^2
(-\det g)
j_D(g,D_w^{-1})D_y^{-2}(D_s^{-1}-w+1)
\xi_D(u,s;y+1,w-1;g) \\
  &=
  -\frac{1}{\det g}
  D_y^{-2}((D_s^{-1}-w+1)
  j_D(g,D_w^{-1})+cD_w^{-1})
  \xi_D(u,s;y+1,w-1;g)\\
  &=
  -\frac{1}{\det g}
  D_y^{-2}(D_s^{-1}-w+1)
  j_D(g,D_w^{-1})\xi_D(u,s;y+1,w-1;g)
  \\
  &\qquad
  -\frac{1}{\det g}
  cD_y^{-2}D_w^{-1}
  \xi_D(u,s;y+1,w-1;g).\label{eq:1.7}
\end{aligned}
\end{multline}
Moreover, 
substituting \eqref{eq:rem-4-2} into the right-hand side of \eqref{eq:1.7}, we obtain
\begin{multline}\label{eq:B=C+C}
j_D(g^{-1},D_y^{-1})D_w^{-2}(D_u^{-1}-y+1)\xi_D(u,s;y,w;g)
\\
\begin{aligned}
  &=
  -\frac{1}{\det g}
  D_y^{-2}(D_s^{-1}-w+1)
  (j_N(g,D_w^{-1})\xi_{D}(u,s;y+2,w-1;g)+(y+1)^{-u}(w-1)^{-s})
  \\
  &\qquad
  -\frac{1}{\det g}
  cD_y^{-2}D_w^{-1}
  \xi_D(u,s;y+1,w-1;g).
\end{aligned}
\end{multline}
In particular, setting $(y,w,g_\eta)=(1,1,g_\eta)$, we obtain
\begin{equation*}
    \xi_D(u-1,s;1,-1;g_\eta)=\xi_D(u,s-1;1,0;g_\eta)-\xi_D(u,s-1;1,-1;g_\eta),
\end{equation*}
namely,
\begin{equation}
    \eta(u,s-1)=\widetilde{\xi}(u,s-1)+\widetilde{\xi}(u-1,s). \label{FR-BC}
\end{equation}
This can be regarded as an interpolation formula of \eqref{1-5-3} (see Example \ref{Exam-6-4}). 
\end{example}

\begin{remark}
  If $g1=1$, then the analytic properties of $\xi_D(u,s;y,w;g)$ in $u,s$ drastically change because the two paths of the integral can not be replaced by the Hankel contours 
due to the singularites of the integrand near the origin in $t,x$.
In this case, by use of the technique employed in the case of multiple zeta functions (see \cite{Komori2010}), we see that $\xi_D(u,s;y,w;g)$ has possible singularities on the hyperplanes $s+u\in\mathbb{Z}$.
\end{remark}

\section{Poly-Bernoulli polynomials associated with $\GL_2(\mathbb{C})$}\label{sec-6}

In this section, let $g \in \GL_2(\mathbb{C})$ satisfying \eqref{eq:def_cond} and $g1\neq 1$. 
We generalize the poly-Bernoulli polynomials from the result in Theorem \ref{AC-uyw}.

\begin{definition}\label{g-Bernoulli}\ \ 
For $u,y,w\in \mathbb{C}$ except for appropriate branch cuts, we define the poly-Bernoulli polynomials $\{\bb_m^{(u)}(y,w;g)\}$ associated with $g$ by
\begin{equation}
\bb_m^{(u)}(y,w;g)=\xi_D(u,-m;y,w;g)\quad (m\in \mathbb{Z}_{\geq 0}). \label{eq:val-xi1}
\end{equation}
\end{definition}

In particular when $g1\neq\infty$, 
it follows from Lemma \ref{lm:IRLT} and Theorem \ref{thm:domain3} that
\begin{equation}
\label{eq:xi-rep}
  \xi_D(u,s;y,w;g)=\frac{1}{\Gamma(s)(e^{2\pi is}-1)}
  \int_H 
  {t^{s-1}e^{-wt}}\frac{\Phi(ge^t,u,y) }{j_D(g,e^t)}dt
\end{equation}
for $\Re u>\nu_1$, $s\in\mathbb{C}$, $\Re y>\nu_\infty$ and $\Re w>\mu_\infty-1$. 
Let $s \to -m\in \mathbb{Z}_{\leq 0}$. Then we obtain the following result. 

\begin{theorem}\label{Th-6-2}\ 
If $g1\neq\infty$, then 
\begin{equation}
e^{wt}\frac{\Phi(ge^{-t},u,y)}{j_D(g,e^{-t})}=\sum_{m=0}^\infty \bb_m^{(u)}(y,w;g)\frac{t^m}{m!}. \label{P-Ber-g}
\end{equation}
for $u,y,w\in \mathbb{C}$ except for appropriate branch cuts.
$\bb_m^{(u)}(y,w;g)$ is a polynomial in $w$.
\end{theorem}

\begin{example}\label{Exam-6-1}\ \ We consider the poly-Bernoulli polynomials defined by 
\begin{equation}
e^{-wt}\frac{\Li_{u}(1-e^{-t})}{1-e^{-t}}=\sum_{m=0}^\infty B_m^{(u)}(w)\frac{t^m}{m!}\quad (u\in \mathbb{C}) \label{P-Ber-poly}
\end{equation}
(see Coppo--Candelpergher \cite{CC2010} in the case $u\in \mathbb{Z}$). 
We define $B_m^{(u)}:=B_m^{(u)}(0)$ and $C_m^{(u)}:=B_m^{(u)}(1)$ which are generalizations of \eqref{1-1} and \eqref{1-2}. Furthermore, we have $B_m^{(1)}(w)=B_m(1-w)=(-1)^mB_m(w)$, where $B_m(w)$ is the classical Bernoulli polynomial. 
From Example \ref{Exam-3-1}, for $g=g_{\eta}=
\begin{pmatrix} -1 & 1 \\ 0 & 1
\end{pmatrix}$, 
we see that the left-hand side of \eqref{P-Ber-g} is equal to that of \eqref{P-Ber-poly} with replacing $-w$ by $w$. 
Hence we have $\bb_m^{(u)}(1,w;g_\eta)=B_m^{(u)}(-w)$. Note that $\bb_m^{(k)}(1,w;g_\eta)$ $(k\in \mathbb{Z})$ coincides with the poly-Bernoulli polynomial defined by Bayad and Hamahata in \cite{BH2011-1}. 
We emphasize that 
\begin{equation}
\begin{split}
\eta(u;-m)&=\xi_D(u,-m;1,0;g_\eta)=\bb_m^{(u)}(1,0;g_\eta)=B_m^{(u)},\\
\tilde{\xi}(u;-m)&=\xi_D(u,-m;1,-1;g_\eta)=\bb_m^{(u)}(1,-1;g_\eta)=C_m^{(u)}
\end{split}
\label{eta-xi-vals}
\end{equation}
for $m\in\mathbb{Z}_{\geq 0}$. Hence, from \eqref{dual-tilde-xi}, we obtain \eqref{1-5-2}. 
Further, from \eqref{eq:Tr-xi}, we obtain 
\begin{equation}
\bb_m^{(u)}(0,0;g_\eta)=\bb_{m-1}^{(u-1)}(1,-1;g_\eta)=C_{m-1}^{(u-1)}\quad (m\in\mathbb{Z}_{\geq 1}). \label{C-vals}
\end{equation}
Therefore it follows from \eqref{P-Ber-g} with $(y,w,g)=(0,0,g_\eta)$ that
\begin{align}
\Phi(1-e^{-t},u,0)&=\Li_{u}(1-e^{-t}) =\sum_{m=1}^\infty C_{m-1}^{(u-1)}\frac{t^{m}}{m!}.\label{gf-C-2}
\end{align}
\end{example}

Combining \eqref{Eq-xi-2} with $k=0$ and \eqref{eq:val-xi1}, we obtain the following.

\begin{theorem}\ \ 
For $y,w\in \mathbb{C}$, 
\begin{equation}
\frac{e^{wt}e^{yx}}{j_D(g,e^{-t})}
  \frac{1}{1-(ge^{-t}) e^{x}}=\sum_{k=0}^\infty \sum_{l=0}^\infty \bb_k^{(-l)}(y,w;g)\frac{t^k x^l}{k!l!}. \label{gene-Bernoulli}
\end{equation}
$\bb_k^{(-l)}(y,w;g)$ is a polynomial in $y$ and $w$.
\end{theorem}

$\xi_D(u,s;y,w;g)$ or $\bb_m^{(u)}(y,w;g)$
satisfies simple transformation formulas for $g=hf$ with a general $h\in\GL_2(\mathbb{C})$ and a special $f\in\GL_2(\mathbb{C})$.
\begin{theorem}
\label{thm:simple_trans}
Let $h\in\GL_2(\mathbb{C})$ and $\alpha\in\mathbb{C}\setminus\{0\}$.
\begin{enumerate}
\item For $f=\begin{pmatrix}
    \alpha & 0 \\ 0 & \alpha
  \end{pmatrix}$,
  \begin{equation}
    \xi_D(u,s;y,w;hf)=\frac{1}{\alpha}\xi_D(u,s;y,w;h).
  \end{equation}
\item For $f=\begin{pmatrix}
    0 & 1 \\ 1 & 0
  \end{pmatrix}$, which corresponds to the inversion $T\mapsto 1/T$,
  \begin{equation}\label{Inverse}
    \bb_m^{(u)}(y,w;hf)=(-1)^m \bb_m^{(u)}(y,-w-1;h).
  \end{equation}
\end{enumerate}
\end{theorem}
\begin{proof}
  The first statement follows directly from the definition.
We show the second statement.
\begin{multline*}
    \bb_m^{(u)}(y,w;hf)
\\
\begin{aligned}
    &=\xi_D(u,-m;y,w;hf)
    \\
    &=
    \frac{(-1)^m m!}{2\pi i}\frac{1}{\Gamma(u)(e^{2\pi iu}-1)}
    \int_{H} dx
    \int_{|t|=\epsilon e^{{}-\Re x}} dt 
    \frac{t^{-m-1}x^{u-1}e^{-wt}e^{-yx}}{j_D(hf,e^t)}
    \frac{1}{1-(hfe^t) e^{-x}}
    \\
    &=
    \frac{(-1)^m m!}{2\pi i}\frac{1}{\Gamma(u)(e^{2\pi iu}-1)}
    \int_{H} dx
    \int_{|t|=\epsilon e^{{}-\Re x}} dt 
    \frac{t^{-m-1}x^{u-1}e^{-wt}e^{-yx}}{j_D(h,e^{-t})j_D(f,e^t)}
    \frac{1}{1-(he^{-t}) e^{-x}}
    \\
    &=
    \frac{(-1)^m m!}{2\pi i}\frac{1}{\Gamma(u)(e^{2\pi iu}-1)}
    \int_{H} dx
    \int_{|t|=\epsilon e^{{}-\Re x}} dt 
    \frac{t^{-m-1}x^{u-1}e^{-(w+1)t}e^{-yx}}{j_D(h,e^{-t})}
    \frac{1}{1-(he^{-t}) e^{-x}}
    \\
    &=(-1)^{m+1-1}
    \frac{(-1)^m m!}{2\pi i}\frac{1}{\Gamma(u)(e^{2\pi iu}-1)}
    \\
    &\qquad\times
    \int_{H} dx
    \int_{|v|=\epsilon e^{{}-\Re x}} dv 
    \frac{v^{-m-1}x^{u-1}e^{(w+1)v}e^{-yx}}{j_D(h,e^v)}
    \frac{1}{1-(he^v) e^{-x}}
    \\
    &=(-1)^m \bb_m^{(u)}(y,-w-1;h),
  \end{aligned}
\end{multline*}
where we changed variables as $v=-t$.
\end{proof}

\begin{example}\label{Exam-6-gene} \ Consider $g_\eta$ and $g_\xi$ in Example \ref{Exam-3-1}. Since $g_\xi=g_\eta f$ for $f=\begin{pmatrix}
    0 & 1 \\ 1 & 0
  \end{pmatrix}$, we have from \eqref{Inverse} that 
\begin{equation}\label{Inverse-B}
    \bb_m^{(u)}(y,w;g_\xi)=(-1)^m \bb_m^{(u)}(y,-w-1;g_\eta).
\end{equation}
Therefore, from \eqref{eta-xi-vals}, we obtain 
\begin{equation}
\xi(u;-m)=\bb_m^{(u)}(1,0;g_\xi)=(-1)^m \bb_m^{(u)}(1,-1;g_\eta)=(-1)^m C_m^{(u)} \label{xi-vals}
\end{equation}
for $m\in \mathbb{Z}_{\geq 0}$, which includes \eqref{1-10}. Hence, by \eqref{xi-dual-2} and \eqref{1-5-2}, we obtain 
\begin{equation}
\check{\xi}(-l;-m)=\xi(-m-1;-l+1)=(-1)^{l-1}C_{l-1}^{(-m-1)}=(-1)^{l-1}C_m^{(-l)}\label{Ber-relation-2}
\end{equation}
for $l\in \mathbb{Z}_{\geq 1}$ and $m\in \mathbb{Z}_{\geq 0}$. It follows from \eqref{xi-vals} and \eqref{Ber-relation-2} that \eqref{xi-dual-2} is an interpolation formula of \eqref{1-5-3}.
\end{example}

\begin{theorem}[Difference relations]
\label{thm:DR}
For $g=
\begin{pmatrix} a & b \\ c & d
\end{pmatrix}$, 
  \begin{multline}
    a\bb_m^{(u)}(y+1,w-1;g)+b\bb_m^{(u)}(y+1,w;g)\\=
    c\bb_m^{(u)}(y,w-1;g)+d\bb_m^{(u)}(y,w;g)-y^{-u}w^{-s}
\label{eq:B-C}
  \end{multline}
holds for $u,y,w\in\mathbb{C}$ except for appropriate branch cuts. 
\end{theorem}
\begin{proof}
  Letting $s=-m\in \mathbb{Z}_{\leq 0}$ in Theorem \ref{thm:trans1} and using Theorem \ref{Th-6-2}, we obtain the assertion.
\end{proof}

\begin{example}\label{Exam-6-4} \ 
It follows from \eqref{eta-xi-vals} and \eqref{C-vals} that \eqref{eq:B-C} with $(y,w,g)=(0,0,g_\eta)$ gives
\begin{equation}
B_m^{(u)}=C_{m}^{(u)}+C_{m-1}^{(u-1)} \label{B-C-relation}
\end{equation}
(see \cite[Section 3]{AK1999} with $u\in \mathbb{Z}$). 
It is to be noted that \eqref{FR-BC} with $s=-m+1$ implies \eqref{B-C-relation}.
\end{example}

Next we prove the duality relations for poly-Bernoulli polynomials associated with $g$ which include ordinary duality relations \eqref{1-5} and \eqref{1-5-2}. Let ${n \brack m}$ $(n,m\in \mathbb{Z}_{\geq 0})$ be the Stirling numbers of the first kind defined by 
\begin{align*}
& {0 \brack 0}=1,\quad {0 \brack m}=0\ (m\geq 1),\quad 
 \prod_{j=0}^{n-1}(X+j)=\sum_{m=0}^n {n \brack m}X^m\ (n\geq 1).
\end{align*}
Note that
\begin{align*}
& \prod_{j=0}^{n-1}(X-j)=\sum_{m=0}^n (-1)^{n+m}{n \brack m}X^m\ (n\geq 1).
\end{align*}

\begin{theorem}[Duality relations]\label{Th-D-F}\ \ 
Let $g=
\begin{pmatrix} a & b \\ c & d
\end{pmatrix}$. 
For $k,m,n\in \mathbb{Z}_{\geq 0}$ and $y,w \in \mathbb{C}$, 
\begin{align}
& \sum_{\tau=0}^{n}\binom{n}{\tau}(-c)^\tau a^{n-\tau}\sum_{j=0}^n {n \brack j}\sum_{\sigma=0}^{j}\binom{j}{\sigma}(\tau-y+1)^{j-\sigma}\bb_m^{(-k-\sigma)}(y-\tau,w-n-1;g) \notag \\
& =\frac{(-1)^{n+1}}{{\rm det}\,g} \sum_{\tau=0}^{n}\binom{n}{\tau}c^\tau d^{n-\tau}\sum_{j=0}^n {n \brack j}\sum_{\sigma=0}^{j}\binom{j}{\sigma}(\tau-w+1)^{j-\sigma}\bb_k^{(-m-\sigma)}(w-\tau,y-n-1;g^{-1}). \label{Dual-F-1}\\
& \sum_{\tau=0}^{n}\binom{n}{\tau}d^\tau (-b)^{n-\tau}\sum_{j=0}^n (-1)^{j}{n \brack j}\sum_{\sigma=0}^{j}\binom{j}{\sigma}(\tau-y-1)^{j-\sigma}\bb_m^{(-k-\sigma)}(y+1-\tau,w;g) \notag \\
& =\frac{(-1)^{n+1}}{{\rm det}\,g} \sum_{\tau=0}^{n}\binom{n}{\tau}a^\tau b^{n-\tau}\sum_{j=0}^n (-1)^{j}{n \brack j}\sum_{\sigma=0}^{j}\binom{j}{\sigma}{(\tau-w-1)}^{j-\sigma}\notag \\
& \qquad \times \bb_k^{(-m-\sigma)}(w+1-\tau,y;g^{-1}), \label{Dual-F-2}\\
& \sum_{\tau=0}^{n}\binom{n}{\tau}d^\tau (-b)^{n-\tau}\sum_{j=0}^n {n \brack j}\sum_{\sigma=0}^{j}\binom{j}{\sigma}(\tau-y+1)^{j-\sigma}\bb_m^{(-k-\sigma)}(y-\tau,w;g) \notag \\
& =\frac{(-1)}{{\rm det}\,g} \sum_{\tau=0}^{n}\binom{n}{\tau}c^\tau d^{n-\tau}\sum_{j=0}^n (-1)^{j}{n \brack j}\sum_{\sigma=0}^{j}\binom{j}{\sigma}(\tau-1-w)^{j-\sigma}\notag\\
& \qquad \times \bb_k^{(-m-\sigma)}(w+1-\tau,y-n-1;g^{-1}). \label{Dual-F-3}
\end{align}
In particular when $n=0$, 
\begin{equation}
\bb_k^{(-m)}(y,w-1;g)=-\frac{1}{{\rm det}\,g}\bb_m^{(-k)}(w,y-1;g^{-1}). \label{Dual-F-4}
\end{equation}
\end{theorem}

\begin{proof}
First we assume that $\Re w$ and $\Re y$ are sufficiently large. By \eqref{dual-1}, we obtain
\begin{align*}
& \sum_{\tau=0}^{n}\binom{n}{\tau}(-c)^\tau a^{n-\tau}\sum_{j=0}^n {n \brack j}\sum_{\sigma=0}^{j}\binom{j}{\sigma}(1-y+\tau)^{j-\sigma}\xi_D(u-\sigma,s;y-\tau,w-n-1;g) \notag \\
& =\frac{(-1)^{n+1}}{{\rm det}\,g} \sum_{\tau=0}^{n}\binom{n}{\tau}c^\tau d^{n-\tau}\sum_{j=0}^n {n \brack j}\sum_{\sigma=0}^{j}\binom{j}{\sigma}(1-w+\tau)^{j-\sigma}\xi_D(s-\sigma,u;w-\tau,y-n-1;g^{-1}).
\end{align*}
It is noted that, for example, $D_y^{-1}$ and $D_u^{-1}$ are commutative and $D_y^{-\tau}(D_u^{-1}-y+k)=(D_u^{-1}-(y-\tau)+k)D_y^{-\tau}$. 
Letting $(u,s)=(-k,-m)$, we obtain from \eqref{eq:val-xi1} that \eqref{Dual-F-1} holds for $y,w\in \mathbb{C}$ if $\Re y$ and $\Re w$ are sufficiently large. Since $\bb_m^{(-k)}(y,w;g)$ is a polynomial in $y,w$, we see that \eqref{Dual-F-1} holds for all $y,w\in \mathbb{C}$. Similar argument works well for \eqref{Dual-F-2} and \eqref{Dual-F-3} by considering \eqref{dual-2} and \eqref{dual-3}, respectively. When $n=0$, each equation gives \eqref{Dual-F-4}. This completes the proof.
\end{proof}

\begin{example}\label{Exam-6-6}
Let 
$(y,w,g)=(1,1,g_\eta)$ in \eqref{Dual-F-1}. 
Then, from Examples \ref{Exam-6-1}, we obtain 
\begin{align}
& \sum_{j=0}^n {n \brack j}B_m^{(-k-j)}(n)=\sum_{j=0}^n {n \brack j}B_k^{(-m-j)}(n), \label{Dual-KST}
\end{align}
which was given by Kaneko, Sakurai and the second-named author (see \cite{KST2016}). In particular when $n=0$ and $1$, we obtain \eqref{1-5} and \eqref{1-5-2}. 
Hence we can regard \eqref{dual-1}--\eqref{dual-3} in Theorem \ref{thm:trans2} as interpolation formulas of the duality relations \eqref{1-5} and \eqref{1-5-2} and their generalizations. Therefore we can give more general examples. For $\alpha\in \mathbb{C}$, let $g=g_{\alpha}=
\begin{pmatrix} -1 & \alpha \\ 0 & 1
\end{pmatrix}$. 
Suppose $\Re \alpha< 2$ 
and let $(y,w)=(1,1)$ in \eqref{P-Ber-g}. Then $g1=\alpha-1\not\in \{1,\infty\}$ and 
\begin{equation}
e^{wt}\frac{\Li_u(\alpha-e^{t})}{\alpha-e^t}=\sum_{m=0}^\infty \bb_m^{(u)}(1,w;g_\alpha)\frac{t^m}{m!}. \label{P-Ber-galpha}
\end{equation}
We have ${\rm det}\,g_\alpha=-1$ and $g_\alpha^{-1}=g_\alpha$. By \eqref{Dual-F-1} with $g_\alpha$, we have
\begin{equation}
 \sum_{j=0}^n {n \brack j}\bb_m^{(-k-j)}(1,-n;g_\alpha)=\sum_{j=0}^n {n \brack j}\bb_k^{(-m-j)}(1,-n;g_\alpha). \label{Dual-alpha}
\end{equation}
Note that \eqref{Dual-alpha} holds for $\alpha\in \mathbb{C}\setminus \{2\}$. 
In fact, $\bb_m^{(-k)}(1,-n;g_\alpha)$ is a rational function in $\alpha$ and continuous for $\alpha \in \mathbb{C}\setminus \{2\}$, because the left-hand side of \eqref{P-Ber-galpha} is analytic around $t=0$ when $\alpha\in \mathbb{C}\setminus \{2\}$. 
In particular, 
\begin{equation}
 \bb_m^{(-k)}(1,0;g_\alpha)=\bb_k^{(-m)}(1,0;g_\alpha). \label{Dual-alpha-2}
\end{equation}
For example, when $\alpha=3,\,-2$ and $\sqrt{-1}$, then we can check that 
\begin{align*}
& \bb_2^{(-3)}(1,0;g_{{3}})=\bb_3^{(-2)}(1,0;g_{{3}})= 242,\\
& \bb_2^{(-3)}(1,0;g_{{-2}})=\bb_3^{(-2)}(1,0;g_{{-2}})= -\frac{1}{512},\\
& \bb_2^{(-3)}(1,0;g_{\sqrt{-1}})=\bb_3^{(-2)}(1,0;g_{\sqrt{-1}})= -\frac{4}{125} - \frac{22}{125}\sqrt{-1}.
\end{align*}
\end{example}

\begin{example}\label{Exam-C-dual}
By \eqref{Dual-F-4} with $(y,w,g)=(-l,-l,g_\eta)$ for $l\in \mathbb{Z}_{\geq 0}$, we have 
\begin{equation}
\bb_k^{(-m)}(-l,-l-1;g_\eta)=\bb_m^{(-k)}(-l,-l-1;g_\eta)\quad (k,m\in \mathbb{Z}_{\geq 0}). \label{Dual-F-4-2}
\end{equation}
Since
\begin{align*}
&\Phi(z;-k,-l)=\sum_{n=0}z^n(n-l)^k=\sum_{i=0}^{l-1}z^i (i-l)^k+z^l \Li_{-k}(z),
\end{align*}
we obtain from \eqref{gf-C-2} that 
\begin{align*}
e^{-(l+1)t}\Phi(1-e^{-t};-k,-l)& = \sum_{i=0}^{l-1}(i-l)^k\sum_{j=0}^{i}\binom{i}{j}(-1)^j e^{-(l+j+1)t}\\
& \quad +\sum_{j=0}^{l}(-1)^j e^{-(l+j+1)t}\sum_{n=1}^\infty C_{n-1}^{(-k-1)}\frac{t^n}{n!}.
\end{align*}
Hence, by \eqref{P-Ber-g}, we have
\begin{align*}
\bb_m^{(-k)}(-l,-l-1;g_\eta)& = (-1)^m\sum_{i=0}^{l-1}(i-l)^k\sum_{j=0}^{i}\binom{i}{j}(-1)^j(l+j+1)^m \\
& \quad +\sum_{i=0}^{m}\binom{m}{i}\sum_{j=0}^{l}(-1)^j(-l-j-1)^{m-i}C_{i-1}^{(-k-1)}.
\end{align*}
Therefore, for example, \eqref{Dual-F-4-2} in the cases $l=0,1$ give new duality relations
\begin{align}
& \sum_{i=1}^{m}\binom{m}{i}(-1)^{m-i}C_{i-1}^{(-k-1)}=\sum_{i=1}^{k}\binom{k}{i}(-1)^{k-i}C_{i-1}^{(-m-1)}, \label{C-dual1}\\
& (-1)^{k+m}2^m+\sum_{i=1}^{m}\binom{m}{i}\left\{(-2)^{m-i}-(-3)^{m-i}\right\}C_{i-1}^{(-k-1)}\notag \\
& \quad =(-1)^{k+m}2^k+\sum_{i=1}^{k}\binom{k}{i}\left\{(-2)^{k-i}-(-3)^{k-i}\right\}C_{i-1}^{(-m-1)} \label{C-dual2}
\end{align}
for $k,m\in \mathbb{Z}_{\geq 1}$. 

By \eqref{Dual-F-4} with $(y,w,g)=(-l,l,g_\eta)$ for $l\in \mathbb{Z}_{\geq 0}$, we obtain
$$\bb_k^{(-m)}(-l,l;g_\eta)=\bb_m^{(-k)}(l+1,-l-1;g_\eta). $$
Similar to the above consideration, this produces new duality relations among $C_{m}^{(-k)}$ different from the above formulas. For example, the case $l=0$ implies \eqref{1-5-2}, and the case $l=1$ gives a new formula 
\begin{equation}
\sum_{j=1}^{k-1}\binom{k}{j}C_{j-1}^{(-m-1)}=\sum_{j=0}^{m}\binom{m}{j}\frac{C_{m-j}C_{j+1}^{(-k)}}{j+1}\quad (k,m\in \mathbb{Z}_{0}). \label{dual-C-l-1}
\end{equation}
\end{example}

Finally, we give certain explicit expressions of poly-Bernoulli polynomials.

\begin{lemma}\label{Lem-6-10}\ \ Assume $g1=0$. 
For $m\in\mathbb{Z}_{\geq 0}$ and $u,y,w\in \mathbb{C}$ except for appropriate branch cuts, 
\begin{equation}
\bb_m^{(u)}(y,w;g)
=\xi_{2,m+1}(u,-m;y,w;g). \label{xi-2-exp}
\end{equation}
\end{lemma}

\begin{proof}
Since $g1=0$, we have $O(ge^t)=O(t)$ ($t \to 0$). Substitute $s=1-k$ $(k\in \mathbb{Z}_{\geq 1})$ into \eqref{Eq-xi-2}. Then the first term on the right-hand side of \eqref{Eq-xi-2} vanishes, because its integrand is holomorphic in $t$ around the origin. 
Hence, from Theorem 
\ref{Th-6-2}, 
we see that \eqref{xi-2-exp} holds for $u,y,w\in \mathbb{C}$ except for appropriate branch cuts. Replacing $m=k-1$, we have the assertion.
\end{proof}

Combining Lemmas \ref{lm:xi_secondp} and \ref{Lem-6-10}, we have the following.

\begin{example}\label{Exam-6-11}\ 
Let $g=
\begin{pmatrix} a & b \\ c & d
\end{pmatrix}$. 
First we assume $g1=0$ and $g\infty=\infty$, namely, $a+b=0$ and $c=0$. 
By Theorem \ref{thm:simple_trans}, we have only to consider $g=h_d:=\begin{pmatrix} -1 & 1 \\ 0 & d
\end{pmatrix}$ 
for $d\in \mathbb{C}\setminus \{0\}$. Note that $h_1=g_\eta$ (see Example \ref{Exam-3-1}). Combining Lemma \ref{lm:xi_secondp} with $k=m+1$, Theorem \ref{Th-6-2} and Lemma \ref{Lem-6-10}, we have
\begin{align}
\bb_m^{(u)}(y,w;h_d)& =\sum_{n=0}^{m}\frac{1}{(y+n)^u}(1-D_w^{-1})^n \frac{w^m}{d^{n+1}}\notag\\
& =\sum_{n=0}^{m}\frac{1}{(y+n)^u}\sum_{j=0}^{n}\binom{n}{j}(-1)^j \frac{(w-j)^m}{d^{n+1}}.\label{Ber-expr-01}
\end{align}
In particular when $(d,y,w)=(1,1,0)$, from Example \ref{Exam-6-1}, we obtain the well-known expression
$$B_m^{(u)}=(-1)^m \sum_{n=0}^{m}\frac{(-1)^n n!}{(n+1)^u}{m \brace n}$$
(see \cite[Theorem 1]{Kaneko1997}), 
where ${m \brace n}$ is the Stirling number of the second kind determined by 
$${m \brace n}=\frac{(-1)^n}{n!}\sum_{j=0}^n (-1)^j \binom{n}{j}j^m\quad (m,n\in \mathbb{Z}_{\geq0}).$$
Next we assume $g1=0$ and $g0=\infty$, namely, $a+b=0$ and $d=0$. Hence we consider $g=h_c':=\begin{pmatrix} 1 & -1 \\ c & 0
\end{pmatrix}$ 
for $c\in \mathbb{C}\setminus \{0\}$. Note that $h_1'=g_\xi$ (see Example \ref{Exam-3-1}). Combining Lemma \ref{lm:xi_secondp} with $k=m+1$, Theorem \ref{Th-6-2} and Lemma \ref{Lem-6-10}, we have
\begin{align}
\bb_m^{(u)}(y,w;h_c')& =\sum_{n=0}^{m}\frac{1}{(y+n)^u}(1-D_w^{-1})^n \frac{(w+n+1)^m}{c^{n+1}}\notag\\
& =\sum_{n=0}^{m}\frac{1}{(y+n)^u}\sum_{j=0}^{n}\binom{n}{j}(-1)^j \frac{(w+n+1-j)^m}{c^{n+1}}.\label{Ber-expr-02}
\end{align}
\end{example}

\section{Proofs of Lemmas \ref{lm:boundV} and \ref{lm:boundNV}}\label{sec:pr}\label{sec-7}
\begin{lemma}
\label{lm:boundn1}
Let $N$ be a neighborhood of the origin in $\mathbb{R}_{\geq 0}$.
Let $a(U),b(U),c(U)$ be real continuous functions in $U\in N$ such that $a(U),c(U)>0$ and
$-\sqrt{a(U)c(U)}\leq b(U)<\sqrt{a(U)c(U)}$ 
for all $U\in N$.
Let $0\leq q\leq 1$.
Then there exists $M>0$ such that
\begin{equation}
  F(U,Y)=\frac{a(U)Y^2-2b(U)UY+c(U)U^2}{U^{2-q}Y^q}\geq M
\end{equation}
for all $(U,Y)$ in a sufficiently small neighborhood of the origin in $\mathbb{R}_{\geq0}^2$ unless the denominator vanishes.  
\end{lemma}
\begin{proof}
  We denote $a(U_0),b(U_0),c(U_0)$ by $a,b,c$ respectively for short.

First assume $0<q\leq 1$.
Fix a sufficiently small $U_0>0$.
Then
  \begin{equation}
    \frac{\partial F(U_0,Y)}{\partial Y}=
\frac{a(2-q)Y^2-2b(1-q)U_0Y-cqU_0^2}{U_0^{2-q}Y^{q+1}}=0
  \end{equation}
implies the unique solution
\begin{equation}
\label{eq:Y_01}
    Y_0=AU_0>0
\end{equation}
with
\begin{equation}
  A=\frac{b (1 - q) + \sqrt{b^2 (1 - q)^2 + a c (2 - q) q}}{a (2 - q)}>0.
\end{equation}
Thus we have
\begin{equation}
    F(U_0,Y)\geq F(U_0,Y_0)=2 \frac{a c (2 - q) - b^2 (1 - q) - 
      b\sqrt{b^2 (1 - q)^2 + a c (2 - q) q}}{a (2 - q)^2 A^q}.
\end{equation}
Here
\begin{gather}
  a c (2 - q) - b^2 (1 - q)
  =a c +(ac - b^2) (1 - q)
  \geq ac
  >0,
  \\
  -b\sqrt{b^2 (1 - q)^2 + a c (2 - q) q}\geq
  -|b|\sqrt{a c (1 - q)^2 + a c (2 - q) q}\geq
  -|b|\sqrt{ac}.
\end{gather}
If $a(0)c(0)\neq b(0)^2$, 
then 
\begin{equation}
  F(U_0,Y_0)\geq
2\frac{\sqrt{a c}(\sqrt{a c} -|b|)}{a (2 - q)^2 A^q}
\end{equation}
and
there exists $M>0$ such that
\begin{equation}
  F(U,Y)\geq M  
\end{equation}
for all $(U,Y)$ in a sufficiently small neighborhood of the origin in $\mathbb{R}_{>0}^2$.
If $a(0)c(0)=b(0)^2$, then by the assumption we have $b(0)=-\sqrt{a(0)c(0)}<0$ and $b(U)<0$ for all sufficiently small $U\geq0$.
Then
\begin{equation}
  -b \sqrt{b^2 (1 - q)^2 + a c (2 - q) q}\geq 0
\end{equation}
and
\begin{equation}
  F(U_0,Y_0)\geq\frac{2ac}{a(2-q)^2A^q}.  
\end{equation}
Hence we have the same conclusion.


Next assume $q=0$.
Fix a sufficiently small $U_0>0$.
Then
  \begin{equation}
    \frac{\partial F(U_0,Y)}{\partial Y}=
    2\frac{aY-bU_0}{U_0^{2}}=0
  \end{equation}
implies the unique solution
\begin{equation}
\label{eq:Y_02}
  Y_0=\frac{bU_0}{a}\in\mathbb{R}.
\end{equation}
If $a(0)c(0)\neq b(0)^2$, then 
we have
\begin{equation}
    F(U_0,Y)\geq F(U_0,Y_0)=\frac{ac-b^2}{a}
\end{equation}
and
there exists $M>0$ such that
\begin{equation}
  F(U,Y)\geq M  
\end{equation}
for all $(U,Y)$ in a sufficiently small neighborhood of the origin in $\mathbb{R}_{>0}\times\mathbb{R}_{\geq0}$.
If $a(0)c(0)=b(0)^2$, then by the assumption we have $b(0)=-\sqrt{a(0)c(0)}<0$ and $b(U)<0$ for all sufficiently small $U\geq0$.
Then
\begin{equation}
  F(U_0,Y)\geq F(U_0,0)=c
\end{equation}
for $Y\geq 0$
and we have the same conclusion.
\end{proof}

\begin{lemma}
\label{lm:boundn2}
Let $N$ be a neighborhood of the origin in $\mathbb{R}_{\geq 0}$.
Let $a(U),b(U),c(U)$ be real continuous functions in $U\in N$ such that $a(U),c(U)>0$, $-\sqrt{a(U)c(U)}\leq b(U)<\sqrt{a(U)c(U)}$ 
for all $U\in N\setminus\{0\}$,
$b(0)=\sqrt{a(0)c(0)}$ and
\begin{equation}
  K=\lim_{U\to 0}\frac{a(U)c(U)-b(U)^2}{U^2}>0.
\end{equation}
Let $0\leq q\leq 2$.
Then there exists $M>0$ such that
\begin{equation}
  G(U,Y)=\frac{a(U)Y^2-2b(U)UY+c(U)U^2}{U^{4-q}Y^q}\geq M
\end{equation}
for all $(U,Y)$ in a sufficiently small neighborhood of the origin in $\mathbb{R}_{\geq0}^2$ unless the denominator vanishes.  
\end{lemma}
\begin{proof}
  We denote $a(U_0),b(U_0),c(U_0)$ by $a,b,c$ respectively for short.
Note that $G(U,Y)=U^{-2}F(U,Y)$, where $F(U,Y)$ is given in Lemma \ref{lm:boundn1}.

First assume $0<q<2$.
Fix a sufficiently small $U_0>0$. 
Then $G(U_0,Y)$ attains its minimum at the same $Y_0$ as \eqref{eq:Y_01}, which is also valid for $1\leq q<2$
and
\begin{equation}
    G(U_0,Y)\geq G(U_0,Y_0)=2 \frac{a c (2 - q) - b^2 (1 - q) - 
      b\sqrt{b^2 (1 - q)^2 + a c (2 - q) q}}{a (2 - q)^2 A^q U_0^2}
\end{equation}
with 
\begin{gather}
   a c (2 - q) - b^2 (1 - q)
   =b^2 +(ac - b^2) (2 - q)
  \geq 0,
  \\
  (a c (2 - q) - b^2 (1 - q))^2 - 
  (b\sqrt{b^2 (1 - q)^2 + a c (2 - q) q})^2=
  a c (ac-b^2) (2-q)^2.
\end{gather}
Since
\begin{gather}
B=  a c (2 - q) - b^2 (1 - q)+ 
b\sqrt{b^2 (1 - q)^2 + a c (2 - q) q}\to
2a(0)c(0),\\
A\to \sqrt{c(0)/a(0)}
\end{gather}
as $U_0\to 0$, we have
\begin{equation}
G(U_0,Y)\geq 
  G(U_0,Y_0)=
  2\frac{a c (2-q)^2}{a (2 - q)^2 A^q B}
  \frac{a c -b^2}{U_0^2 }
  \to
  \frac{K}{a(0)\sqrt{c(0)/a(0)}^q}>0
\end{equation}
as $U_0\to 0$.
Thus
there exists $M>0$ such that
\begin{equation}
  G(U,Y)\geq M  
\end{equation}
for all $(U,Y)$ in a sufficiently small neighborhood of the origin in $\mathbb{R}_{>0}^2$.

Secondly assume $q=0$.
Fix a sufficiently small $U_0>0$.
Then $G(U_0,Y)$ attains its minimum at the same $Y_0$ as \eqref{eq:Y_02} and
\begin{equation}
    G(U_0,Y)\geq G(U_0,Y_0)=\frac{ac-b^2}{aU_0^2}\to\frac{K}{a(0)}>0
\end{equation}
as $U_0\to 0$.
Thus
there exists $M>0$ such that
\begin{equation}
  G(U,Y)\geq M  
\end{equation}
for all $(U,Y)$ in a sufficiently small neighborhood of the origin in $\mathbb{R}_{>0}\times\mathbb{R}_{\geq0}$.

Thirdly we assume $q=2$. 
Fix a sufficiently small $U_0>0$.
Then
  \begin{equation}
    \frac{\partial G(U_0,Y)}{\partial Y}=
    2\frac{bU_0Y-cU_0^2}{U_0^{2}Y^{3}}=0
  \end{equation}
implies the unique solution
\begin{equation}
    Y_0=\frac{cU_0}{b}>0
\end{equation}
because
by the assumption, $b(U)>0$ for all sufficiently small $U\geq0$.
Thus we have
\begin{equation}
    G(U_0,Y)\geq G(U_0,Y_0)=\frac{ac-b^2}{cU_0^2}\to\frac{K}{c(0)}>0
\end{equation}
as $U_0\to 0$
and
there exists $M>0$ such that
\begin{equation}
  G(U,Y)\geq M  
\end{equation}
for all $(U,Y)$ in a sufficiently small neighborhood of the origin in $\mathbb{R}_{>0}^2$.
\end{proof}

\begin{lemma}
\label{lm:bound}
Assume that $h=\begin{pmatrix}
  \alpha & \beta \\ \gamma & \delta
\end{pmatrix}\in\GL_2(\mathbb{C})$ satisfies $hU=Y$ for only $(U,Y)=(0,0)$ in a neighborhood of the origin in $\mathbb{R}_{\geq 0}^2$.
Then for $0\leq q\leq1$,
there exists $M>0$ such that
\begin{equation}
  \frac{1}{|\alpha U+\beta-Y(\gamma U+\delta)|}\leq 
  \begin{cases}
    \dfrac{M}{U^{1-q}Y^q}\qquad&\text{if the origin is not a cusp}, \\
    \dfrac{M}{U^{2(1-q)}Y^{2q}}\qquad&\text{if the origin is a cusp}
  \end{cases}
\end{equation}
in a neighborhood of the origin in $\mathbb{R}_{>0}^2$. 
\end{lemma}
\begin{proof}
Assume that $h=\begin{pmatrix}
  \alpha & \beta \\ \gamma & \delta
\end{pmatrix}\in\GL_2(\mathbb{C})$ satisfies $hU=Y$ for only $(U,Y)=(0,0)$ in the neighborhood of the origin in $\mathbb{R}_{\geq 0}^2$,
Then $h0=0$ implies $\beta=0$ and $\det h=\alpha\delta\neq 0$.
Hence $hU=Y$ is rewritten as
\begin{equation}
\label{eq:pos_Y}
  Y=\frac{\alpha U}{\gamma U+\delta}=\frac{(\alpha\overline{\gamma} U+\alpha\overline{\delta}) U}{|\gamma U+\delta|^2}.  
\end{equation}
Assume that 
$\alpha\overline{\delta}\in\mathbb{R}_{>0}$ and 
$\alpha\overline{\gamma}\in\mathbb{R}$.
Then in any neighborhood of the origin, a pair $(U,Y)$ with a small $U>0$ and
 $Y$ given by
\eqref{eq:pos_Y}
is a solution.
Thus if the solution is only $(U,Y)=(0,0)$ in a neighborhood of the origin in $\mathbb{R}_{\geq 0}^2$, then
$\alpha\overline{\delta}\notin\mathbb{R}_{>0}$ or
$\alpha\overline{\gamma}\notin\mathbb{R}$. 

If the origin is a cusp,
then
 $\dfrac{d}{dU} hU\Bigr|_{U=0}=\dfrac{\det h}{\delta^2}=\dfrac{\alpha\overline{\delta}}{|\delta|^2}>0$
and hence $\alpha\overline{\delta}\in\mathbb{R}_{>0}$. The converse is also true.

Assume $0\leq q\leq 1/2$.
Consider
\begin{equation}
    |\alpha U-\delta Y-\gamma UY|^2
    =|\delta+\gamma U|^2
    Y^2-2\Re(\alpha\overline{\delta}+\alpha\overline{\gamma}U)UY+|\alpha|^2 U^2
\end{equation}
and let
\begin{equation}
  a(U)=  |\delta+\gamma U|^2, \qquad
  b(U)=  \Re(\alpha\overline{\delta}+\alpha\overline{\gamma}U),\qquad
  c(U)=  |\alpha|^2.
\end{equation}
We check the assumptions in Lemmas \ref{lm:boundn1} and \ref{lm:boundn2}.
Since $\alpha,\delta\neq 0$, we see that $a(U),c(U)>0$ for all sufficiently small $U\geq 0$. Furthermore
\begin{equation}
  a(U)c(U)-b(U)^2=|\alpha\overline{\delta}+\alpha\overline{\gamma}U|^2-
(\Re(\alpha\overline{\delta}+\alpha\overline{\gamma}U))^2
=
(\Im(\alpha\overline{\delta}+\alpha\overline{\gamma}U))^2\geq0,
\end{equation}
which implies 
$-\sqrt{a(U)c(U)}\leq b(U)\leq\sqrt{a(U)c(U)}$.
Since $\alpha\overline{\delta}\notin\mathbb{R}_{>0}$ or
$\alpha\overline{\gamma}\notin\mathbb{R}$,
\begin{equation}
\sqrt{a(U)c(U)}-b(U)=
  |\alpha\overline{\delta}+\alpha\overline{\gamma}U|-
\Re(\alpha\overline{\delta}+\alpha\overline{\gamma}U)\neq 0
\end{equation}
holds for all sufficiently small $U\geq 0$ if $\alpha\overline{\delta}\notin\mathbb{R}_{>0}$,
and for all sufficiently small $U>0$ if $\alpha\overline{\delta}\in\mathbb{R}_{>0}$. In the latter case, 
\begin{equation}
  b(0)=\Re \alpha\overline{\delta}=|\alpha\overline{\delta}|=\sqrt{a(0)c(0)}
\end{equation}
and
\begin{equation}
  \frac{a(U)c(U)-b(U)^2}{U^2}=
\frac{(\Im(\alpha\overline{\delta}+\alpha\overline{\gamma}U))^2}{U^2}
=
\frac{(\Im \alpha\overline{\gamma}U)^2}{U^2}
=
(\Im \alpha\overline{\gamma})^2>0.
\end{equation}
Thus we have checked the assumptions required and have the assertions in this case.


For $1/2< q\leq 1$, 
exchanging the roles of $U$ and $Y$, and
applying Lemmas \ref{lm:boundn1} and \ref{lm:boundn2}
with
\begin{equation}
    |\alpha U-\delta Y-\gamma UY|^2
=|\alpha-\gamma Y|^2
    U^2-2\Re(\delta\overline{\alpha}-\delta\overline{\gamma}Y)YU+|\delta|^2 Y^2
\end{equation}
and
\begin{equation}
a(Y)=|\alpha-\gamma Y|^2,\qquad
b(Y)=\Re(\delta\overline{\alpha}-\delta\overline{\gamma}Y),\qquad
c(Y)=|\delta|^2,
\end{equation}
we
have the assertions in this case.
Here we used the fact that
$\alpha\overline{\delta}\in\mathbb{R}_{>0}$ implies
$\alpha\overline{\gamma}\notin\mathbb{R}$,
and hence
$\delta\overline{\alpha}\in\mathbb{R}_{>0}$ and
$\delta\overline{\gamma}\notin\mathbb{R}$.
\end{proof}

\begin{lemma}
\label{lm:div}
Assume that $h=\begin{pmatrix}
  \alpha & \beta \\ \gamma & \delta
\end{pmatrix}\in\GL_2(\mathbb{C})$ satisfies $hU=Y$ for only $(U,Y)=(0,0)$ in a neighborhood of the origin in $\mathbb{R}_{\geq 0}^2$.
Then there exists $\epsilon>0$ such that 
\begin{equation}
  \frac{1}{\epsilon}|Y|> |U|>\epsilon |Y|
\end{equation}
for any 
pair $(U,Y)$ satisfying
$hU=Y$ in
 a sufficiently small neighborhood of the origin in $\mathbb{C}^2$.
\end{lemma}
\begin{proof}
From the first paragraph of the proof of Lemma \ref{lm:bound}, we see that
$\beta=0$ and $\alpha\delta\neq 0$.
Since $hU=Y$ is rewritten as
$Y=\dfrac{\alpha U}{\gamma U+\delta}$,
we have
\begin{equation}
  |Y|
\geq\frac{|\alpha|}{|\delta|\Bigl|1+\dfrac{\gamma}{\delta}U\Bigr|}|U|
\geq\frac{|\alpha|}{2|\delta|}|U|.
\end{equation}
Similarly
$U=\dfrac{\delta Y}{\gamma Y-\alpha}$ implies
\begin{equation}
  |U|
\geq\frac{|\delta|}{2|\alpha|}|Y|.
\end{equation}

\end{proof}

\begin{proof}[Proof of Lemma \ref{lm:boundV}]
For $Z\in\{1,\infty\}$, let 
\begin{equation}
  k_Z=
  \begin{pmatrix}
    \Tilde{Z} & -1 \\ -(-1)^{\Tilde{Z}} & 0
  \end{pmatrix}=
  \begin{cases}
  \begin{pmatrix}
    1 & -1 \\ 1 & 0
  \end{pmatrix}\qquad&(Z=1),\\
  \begin{pmatrix}
    0 & -1 \\ -1 & 0
  \end{pmatrix}\qquad&(Z=\infty). 
  \end{cases}
\end{equation}
Note that $k_Z$ maps a neighborhood of $Z$ in $[1,+\infty]$ to a neighborhood of the origin in $\mathbb{R}_{\geq 0}$.

By putting $U=k_{T_0}T$ and $Y=k_{X_0}X$,
we see that $h=k_{X_0}gk_{T_0}^{-1}=
\begin{pmatrix}
  \alpha & \beta \\ \gamma & \delta
\end{pmatrix}$ satisfies the assumption in Lemma \ref{lm:bound}.
Since 
\begin{equation}
  k_Z^{-1}=
  \begin{pmatrix}
    0 & -(-1)^{\Tilde{Z}} \\ -1 & -(-1)^{\Tilde{Z}}\Tilde{Z}
  \end{pmatrix}
\end{equation}
and
\begin{gather}
  j_D(k_{X_0}^{-1}hk_{T_0},k_{T_0}^{-1}U)=\frac{j_D(k_{X_0}^{-1}h,U)}{j_D(k_{T_0}^{-1},U)},\\
j_D(k_{X_0}^{-1}h,U)=j_D(k_{X_0}^{-1},hU)j_D(h,U),\qquad
j_N(k_{X_0}^{-1}h,U)=j_N(k_{X_0}^{-1},hU)j_D(h,U),\qquad
\end{gather}
we have
\begin{equation}
\label{eq:jD}
  \begin{split}
    j_D(g,T)(1-(gT)X^{-1})&=
    j_D(k_{X_0}^{-1}hk_{T_0},k_{T_0}^{-1}U)(1-(k_{X_0}^{-1}hU)(k_{X_0}^{-1}Y)^{-1})\\
    &=
    \frac{j_D(k_{X_0}^{-1}h,U)}{j_D(k_{T_0}^{-1},U)}
    \Bigl(1-\frac{j_N(k_{X_0}^{-1}h,U)}{j_D(k_{X_0}^{-1}h,U)}
    \frac{j_D(k_{X_0}^{-1},Y)}{j_N(k_{X_0}^{-1},Y)}
    \Bigr)
    \\
    &=
    \frac{j_D(k_{X_0}^{-1}h,U)j_N(k_{X_0}^{-1},Y)-j_N(k_{X_0}^{-1}h,U)j_D(k_{X_0}^{-1},Y)}
    {j_D(k_{T_0}^{-1},U)j_N(k_{X_0}^{-1},Y)}
    \\
    &=
    \frac{j_D(k_{X_0}^{-1},hU)-j_D(k_{X_0}^{-1},Y)}
    {j_D(k_{T_0}^{-1},U)}j_D(h,U)
    \\
    &=
    \frac{j_N(h,U)-Yj_D(h,U)}
    {U+(-1)^{\Tilde{T_0}} \Tilde{T_0}}
    \\
    &=
    \frac{\alpha U+\beta-Y (\gamma U+\delta)}
    {U+(-1)^{\Tilde{T_0}} \Tilde{T_0}}.
  \end{split}
\end{equation}
Hence
\begin{multline}
    \Bigl|\frac{1}{j_D(g,T)}\frac{1}{(1-(gT)X^{-1})}\Bigr|\\
    \leq M|U+(-1)^{\Tilde{T_0}} \Tilde{T_0}|
\times
  \begin{cases}
    \dfrac{1}{U^{1-q}Y^q}\qquad&\text{if the vertex is not a cusp}, \\
    \dfrac{1}{U^{2(1-q)}Y^{2q}}\qquad&\text{if the vertex is a cusp}.
  \end{cases}
\end{multline}
Since
\begin{align}
  U=k_{T_0}T&=\frac{\Tilde{T_0}T-1}{-(-1)^{\Tilde{T_0}}T}, \\
  Y=k_{X_0}X&=\frac{\Tilde{X_0}X-1}{-(-1)^{\Tilde{X_0}}X},
\end{align}
we obtain the first result.

The second statement follows from Lemma \ref{lm:div}.
\end{proof}

\begin{proof}[Proof of Lemma \ref{lm:boundNV}]
We use the same notation as in Lemma \ref{lm:boundV}.
If $V(g)\neq\emptyset$, then we fix $X_0\in V(g)$ and $T_0=g^{-1}X_0$,
and otherwise put $X_0=T_0=\infty$.
Further
put $U=k_{T_0}T$, $Y=k_{X_0}X$,
$h=k_{X_0}gk_{T_0}^{-1}=
\begin{pmatrix}
  \alpha & \beta \\ \gamma & \delta
\end{pmatrix}$
and $S(g)=\{(U,Y)\in[0,1]^2~|~hU=Y\}$.
We see that 
\begin{equation}
  S(g)=
  \begin{cases}
    \emptyset\qquad& (\sharp V(g)=0),\\
    \{(0,0)\}\qquad& (\sharp V(g)=1),\\
    \{(0,0),(1,1)\}\qquad& (\sharp V(g)=2),
  \end{cases}
\end{equation}
and
$S(g)$
coincides with the set of all solutions of $\alpha U+\beta=Y(\gamma U+\delta)$ in $[0,1]^2$.
Let $N_{\epsilon'}\subset(k_{T_0}\times k_{X_0})(N)$ be an open $\epsilon'$-neighborhood of $S(g)$ in $\mathbb{C}^2$ and
 $B_{\epsilon''}$ be an $\epsilon''$-neighborhood of $[0,1]$ in $\mathbb{C}$.
Since $[0,1]^2\setminus N_{\epsilon'}$ is a compact set in $\mathbb{C}^2$,  there exists $M>0$ and $\epsilon''>0$
such that
\begin{equation}
  |\alpha U+\beta-Y(\gamma U+\delta)|>\frac{1}{M}
\end{equation}
for all $(U,Y)\in B_{\epsilon''}^2\setminus N_{\epsilon'}$.
By the same calculation as \eqref{eq:jD},
we have
\begin{equation}
  j_D(g,T)(1-(gT)X^{-1})
  =
  \frac{\alpha U+\beta-Y(\gamma U+\delta)}{U+(-1)^{\Tilde{T_0}} \Tilde{T_0}}.
\end{equation}
Hence
\begin{equation}
  \Bigl|\frac{1}{j_D(g,T)}\frac{1}{(1-(gT)X^{-1})}\Bigr|
  \leq \frac{M}{|T|}
\end{equation}
for all $(T,X)\in (k_{T_0}\times k_{X_0})^{-1}(B_{\epsilon''}^2\setminus N_{\epsilon'})\cap\mathbb{C}^2$.
Since $k_1^{-1}(B_{\epsilon''})=k_\infty^{-1}(B_{\epsilon''})\supset W_{1,\epsilon}$ for a sufficiently small $\epsilon>0$, we have 
\begin{equation}
(k_{T_0}\times k_{X_0})^{-1}(B_{\epsilon''}^2\setminus N_{\epsilon'})\cap\mathbb{C}^2
  \supset ((k_{T_0}^{-1}(B_{\epsilon''})\times k_{X_0}^{-1}(B_{\epsilon''}))\setminus N)\cap\mathbb{C}^2
\supset W_{1,\epsilon}^2\setminus N,
\end{equation}
and 
the assertion.
\end{proof}

\ 

\ 

\begin{flushleft}
\begin{small}

{Y. Komori}: 
{Department of Mathematics,
Rikkyo University,
Nishi-Ikebukuro, Toshima-ku,
Tokyo 171-8501, 
Japan}

e-mail: \texttt{komori@rikkyo.ac.jp}

\

{H. Tsumura}: 
{Department of Mathematics and Information Sciences, Tokyo Metropolitan University, 1-1, Minami-Ohsawa, Hachioji, Tokyo 192-0397 Japan}

e-mail: \texttt{tsumura@tmu.ac.jp}
\end{small}
\end{flushleft}

\end{document}